\documentclass[10pt,reqno]{amsart}
\usepackage{latexsym}
\usepackage{amssymb}
\usepackage{amsmath}
\usepackage{mathtools}
\usepackage[english]{babel}
\usepackage[dvips]{graphicx}
\usepackage{texdraw}
\usepackage{graphics}
\usepackage{pgf,tikz}
\usepackage{enumerate}
\usepackage{xcolor}
\font\tenmath=msbm10 \font\sevenmath=msbm7 \font\fivemath=msbm5
\newfam\mathfam \textfont\mathfam=\tenmath
\scriptfont\mathfam=\sevenmath \scriptscriptfont\mathfam=\fivemath

\def \\ { \cr }

\newcommand{\vvert}{|\!|\!|}
\newcommand{\RR}{{\mathbb R}}
\newcommand{\CC}{{\mathbb C}}
\newcommand{\NN}{{\mathbb N}}
\newcommand{\ZZ}{\mathbb Z}
\newcommand{\QQ}{{\mathbb Q}}
\newcommand{\EE}{{\mathbb E}}

\providecommand{\abs}[1]{\left\lvert {#1} \right\rvert}

\def\P{{\mathcal P}}

\def\T{{\mathcal T}}

\newcommand\la{\langle}
\newcommand\ra{\rangle}

\newcommand{\set}[1]{\left\{{#1}\right\}}
\newcommand{\implica}{\Longrightarrow}

\newcommand{\talque}{;}
\newcommand{\y}{ \; \textrm{and} \; }
\newcommand{\ds}{\displaystyle}

\newcommand{\torref}[1]{\tau_{#1}}

\newcommand{\sufijocb}[3]{S_{{#1}}({#3},{#2})}       
\newcommand{\sufijocbb}[4]{{S}_{{#1},{#2}}({#3},{#4})} 

\newcommand{\matrizp}[4]{P_{{#1},{#2}}({#3},{#4})}               

\newcommand{\btheta}{{\bf\theta}}
\newcommand{\source}{{\mathsf{s}}}
\newcommand{\range}{{\mathsf{r}}}
\numberwithin{equation}{section}
\newtheorem{theo}{Theorem}
\newtheorem{prop}[theo]{Proposition}
\newtheorem{coro}[theo]{Corollary}
\newtheorem{lemma}[theo]{Lemma}

\theoremstyle{remark}
\newtheorem{rmrk}{Remark}

\begin{document}

\title
{Eigenvalues of minimal Cantor systems}

\author{Fabien Durand}
\address{Laboratoire Ami\'enois
de Math\'ematiques Fondamentales et Appliqu\'ees, CNRS-UMR 7352, Universit\'{e} de Picardie Jules Verne, 33 rue Saint Leu, 80000 Amiens,
France.} \email{fabien.durand@u-picardie.fr}

\author{Alexander Frank}
\address{Departamento de Ingenier\'{\i}a
Matem\'atica and Centro de Modelamiento Ma\-te\-m\'a\-ti\-co, CNRS-UMI 2807, Universidad de Chile, Beauchef 851, Santiago, Chile.}
\email{afrank@dim.uchile.cl}

\author{Alejandro Maass}
\address{Departamento de Ingenier\'{\i}a
Matem\'atica and Centro de Modelamiento Ma\-te\-m\'a\-ti\-co, CNRS-UMI 2807, Universidad de Chile, Beauchef 851, Santiago,
Chile.}\email{amaass@dim.uchile.cl}


\thanks{The first author was partially supported by the ANR programs FAN and DynA3S, and MathAmSud DYSTIL. The second and third authors were partially supported by grants Basal-CMM PFB-03 and Mecesup UCH 0607. We acknowledge invitations from U. Picardie Jules Verne and CNRS where part of this work was developed.}

\begin{abstract}
In this article we give necessary and sufficient conditions that a complex number must satisfy to be a continuous eigenvalue of a minimal Cantor system. Similarly, for minimal Cantor systems of finite rank, we provide necessary and sufficient conditions for having a measure theoretical eigenvalue.
These conditions are established from the combinatorial information of the Bratteli-Vershik representations of such systems. As an application, from any minimal Cantor system, we construct a strong orbit equivalent system without irrational eigenvalues which shares all measure theoretical eigenvalues with the original system. In a second application a minimal Cantor system is constructed  satisfying the so-called 
maximal continuous eigenvalue group property. 
\end{abstract}

\date{July 4, 2017}

\maketitle

\markboth{Fabien Durand, Alexander Frank, Alejandro Maass}{Eigenvalues of finite rank minimal Cantor systems}

\section{Introduction}

The spectral theory of dynamical systems and, in particular, the study of eigenvalues of topological dynamical systems, either from a measure theoretical or a topological perspective, is a fundamental topic in ergodic theory, which allows one to understand mixing properties and the characterization of the Kronecker and maximal equicontinuous factors. Particularly interesting and rich has been the study of eigenvalues and weakly mixing properties of classical systems like interval exchange transformations \cite{nogueirarudolph, avilaforni, ferenczizamboniI, ferenczizamboniII} or other systems arising from translations on surfaces \cite{aviladelecroix}. From the symbolic dynamics point of view most of these systems have representations as minimal Cantor systems of finite topological rank, {\it i.e.}, there is a symbolic extension that can be represented by a Bratteli-Vershik system such that the number of Kakutani-Rohlin towers per level is globally bounded. To characterize eigenvalues of the original systems it is enough to consider this class of Cantor systems.   
Of course, a general approach like this assumes that the particular nature and information carried by the original dynamics can be effectively translated into concrete properties of a ``good'' Kakutani-Rohlin representation, which is not evident. Nevertheless, good representations for interval exchange transformations and, in particular, irrational rotations of the torus have already been proposed (see \cite{gjerdejohansen,dartnelldurandmaass}). 

With these examples in mind, our main motivation is to provide general necessary and sufficient conditions for a complex number to be the eigenvalue, either continuous or measure theoretical, of a minimal Cantor system of finite topological rank and when possible to get the same kind of results for any minimal Cantor system. In addition, we also want these conditions to be useful for studying the weakly mixing property,  {\it i.e.}, the absence of eigenvalues, or any other question relating eigenvalues with the dynamics of minimal Cantor systems.  

Some problems addressed in this article for different subclasses of minimal Cantor systems of finite topological rank has been considered since the pioneering work of Dekking \cite{dekking} and Host \cite{host}. There, it was stated that measurable eigenvalues of primitive substitution dynamical systems are always associated to continuous eigenfunctions, thus the maximal equicontinuous and measure theoretical Kronecker factors coincide. Implicitly, both works give conditions to be a measurable eigenvalue; however, the complete characterization of eigenvalues for substitution dynamical systems was given in \cite{ferenczimauduitnogueira}. Later, necessary and sufficient conditions to characterize continuous and measurable eigenvalues of linearly recurrent minimal Cantor systems were provided in \cite{lr} and \cite{necesariasuficiente}. These conditions are very effective and rely on the combinatorial data carried by the Bratteli-Vershik representations. Even if linearly recurrent systems are natural from the symbolic dynamics point of view (see \cite{du1,du2}), this class could be considered ``small'', meaning that in many classical cases, like interval exchange transformations, only a few maps have a symbolic representation of this kind. In fact, most of them are of finite topological rank and not linearly recurrent. There are few general results concerning eigenvalues of minimal Cantor systems of finite topological rank. Some preliminary results are given in \cite{rangofinito} and a detailed study of eigenvalues of Toeplitz systems of finite topological rank is given in \cite{dfm}. This last work motivates the ideas of the current work.  

In this article we provide necessary and sufficient conditions that a complex number should satisfy to be a measurable eigenvalue of a minimal Cantor system of finite topological rank (Theorem \ref{theo:cns_sinrho} and Theorem \ref{theo:cnsvpmedible}). In addition, we give a necessary and sufficient condition  
for a complex number to be a continuous eigenvalue of a minimal Cantor system, that is, we succeeded in dropping the finite rank hypothesis (Theorem \ref{theo:nec-suff-cont}). 
In its conception, the conditions are very similar to those proposed for linearly recurrent systems. They are given in the form of the convergence of some series or special sequences and only depend on the combinatorial data provided by the  Bratteli-Vershik representations. The main difference here is that we need to include in an algebraic way the information of the local orders carried by these representations. Thus, the drawback of these conditions is that they depend on a non trivial computation. 

To illustrate the use of the conditions provided in this article we consider different examples and applications. 

First we prove that our conditions extend the results in \cite{dfm} to characterize eigenvalues of finite rank Toeplitz minimal systems. This class, even if simple, allows to see the amount of information needed to compute eigenvalues using the proposed conditions. Then, a first application relates the notions of continuous eigenvalues and strong orbit equivalence. We use our necessary and sufficient condition in the continuous case to prove that, by doing controlled modifications of the local orders of a Bratteli-Vershik system, one can alter the group of continuous eigenvalues. In particular, starting from a minimal Cantor system without roots of unity as continuous eigenvalues we produce a strong orbit equivalent system that is topologically weakly mixing and which shares the Kronecker factor with the original system for any ergodic measure. In 
\cite{GHH} a similar result is obtained but without the control on the non continuous eigenvalues and in 
\cite{sadunpriebe} a similar example is developed in the context of tiling systems. In a second example, the conditions to be measurable eigenvalues and previous application are used to construct a topologically weakly mixing minimal Cantor system of rank two admitting all rational numbers as measure theoretical eigenvalues, showing that topological rank is not an obstruction to have non continuous rational eigenvalues as in the Toeplitz case. 
Finally, inspired by questions in \cite{Cortez&Durand&Petite:2016} and \cite{GHH}, we use our main theorems 
to produce an expansive minimal Cantor system whose group of continuous eigenvalues coincides with the intersection of the images of the so-called group of traces. 

The article is organized as follows. In Section \ref{sec:definitions} we provide the main definitions concerning eigenvalues of dynamical systems and Bratteli-Vershik representations. Section \ref{sec:continuous} is devoted to the main result in the continuous case (Theorem \ref{theo:nec-suff-cont}). In this section we do not use the finite rank hypothesis. Section \ref{sec:measurable} is focused on the main results in the measurable case (Theorem \ref{theo:cns_sinrho} and Theorem \ref{theo:cnsvpmedible}). These results only concern minimal Cantor systems of finite topological rank. Finally, in Section \ref{examples} we develop examples and applications illustrating our main results. 

\section{Definitions and notation}
\label{sec:definitions}

\subsection{Dynamical systems and eigenvalues}
A \emph{topological dynamical system}, or just \emph{dynamical system}, is a compact Hausdorff space $X$ together with a homeomorphism 
$T:X\rightarrow X$. We use the notation $\left( X,T\right)$. If $X$ is a Cantor set 
({\it i.e.}, $X$ has a countable basis of closed and open sets and it has no isolated points) we say that the system is Cantor. A dynamical system is \emph{minimal} if all orbits are dense in $X$,
or equivalently if the only non empty closed invariant set is $X$.

A complex number $\lambda$ is a {\it continuous eigenvalue} of $(X,T)$ if there exists a continuous function $f : X\to \CC$, $f\not = 0$, such
that $f\circ T = \lambda f$; $f$ is called a {\it continuous eigenfunction} (associated to $\lambda$). The system $(X,T)$ is \emph{topologically weakly mixing} if it has no non constant continuous eigenfunctions.
Let $\mu$ be a $T$-invariant probability measure defined on the Borel $\sigma$-algebra of $X$, {\it i.e.}, $T\mu = \mu$. A complex number $\lambda$ is an {\it eigenvalue} of the
dynamical system $(X,T)$ with respect to $\mu$ if there exists $f\in L^2 (X,\mu)$, $f\not = 0$, such that $f\circ T = \lambda f$; $f$ is
called an {\it eigenfunction} (associated to $\lambda$). If $\mu$ is ergodic, then every eigenvalue for $\mu$ has modulus 1 and every eigenfunction has a constant modulus $\mu$-almost surely. Of course, continuous eigenvalues are eigenvalues for $\mu$. The system is \emph{weakly mixing} for $\mu$ if it has no non constant eigenfunctions.

If $\lambda=\exp(2i\pi\alpha)$ is either a continuous or measurable eigenvalue with $\alpha$ an irrational number we say that $\lambda$ is an irrational eigenvalue; in the case $\alpha$ is rational we say  that $\lambda$ is a rational eigenvalue.

\subsection{Bratteli-Vershik representations}
Let $(X,T)$ be a minimal Cantor system. It can be represented by an ordered Bratteli diagram together with the Vershik transformation acting on it. 
This couple is called a Bratteli-Vershik representation of the system. We give a brief outline of this construction emphasizing the notation in this paper. For details on this theory see \cite{hps} or \cite{review}.

\subsubsection{Bratteli diagrams}
A Bratteli diagram is an infinite graph $\left( V,E\right)$ which consists of a vertex set $V$ 
and an edge set $E$, both of
which are divided into levels $V=V_{0}\cup V_{1}\cup \ldots$ and $E=E_{1}\cup E_{2}\cup \ldots$, 
where all levels are pairwise disjoint. 
The set $V_{0}$ is a singleton $\{v_{0}\}$ and for all $n\geq 1$ edges in $E_{n}$ join vertices in $V_{n-1}$ to vertices in $V_{n}$. If $e\in E$ connects $u\in V_{n-1}$ with $v \in V_n$ we write $\source(e)=u$ and $\range(e)=v$, where $\source:E_n\to V_{n-1}$ and $\range:E_n\to V_{n}$ are the source and range maps, respectively. 
It is also required that $\source^{-1}(v)\not = \emptyset$ for all $v \in V$ and that 
$\range^{-1}(v)\not = \emptyset$ for all $v \in V\setminus V_0$.
For all $n\geq 1$
we set $\#V_{n}=d_{n}$ and we write $V_n=\{1,\ldots,d_n\}$ to simplify notation.   

Fix $n\geq 1$. We call \emph{level} $n$ of the diagram the subgraph consisting of the vertices in $V_{n-1}\cup V_{n}$ 
and the edges $E_{n}$ between these vertices. Level $1$ is called the \emph{hat} of the Bratteli diagram. 
We describe the edge set $E_n$ using a $V_{n-1}\times V_{n}$ incidence matrix $M_{n}$ for which its $(u,v)$ entry is the number of edges in $E_{n}$ joining vertex $u \in V_{n-1}$ with vertex $v \in V_{n}$. We also set $P_{n}=M_{2}\cdots M_{n}$, with the convention that $P_{1}=I$, where $I$ denotes the identity matrix. The number of paths joining $v_{0} \in V_{0}$ and a vertex $v\in V_{n}$ is given by coordinate $v$ of the \emph{height row vector} $h_{n}=\left(h_n(u) ; u \in V_{n} \right) \in \NN ^{V_{n}}$. Notice that $h_{1}=M_{1}$ and $h_{n}=h_{1}P_{n}$. 

We also consider several levels at the same time. 
For integers $0 \leq m< n$ we denote by $E_{m,n}$ the set of all paths in the graph joining vertices of  
$V_{m}$ with vertices of $V_{n}$. We define matrices $P_{m,n}=M_{m+1} \cdots M_{n}$, with the convention that $P_{n,n}=I$
for $1\leq m \leq n$. Clearly, entry $\matrizp{m}{n}{u}{v}$ of matrix $P_{m,n}$ is the number of paths in $E_{m,n}$ from vertex $u \in V_{m}$ to vertex $v \in V_{n}$. 
It can be easily checked that $h_{n}=h_{m}P_{m,n}$. 

A Bratteli diagram $(V,E)$ is called \emph{simple} if for any $m\geq 1$ there exists $n > m$ such that each pair of vertices $u\in V_m$ and $v \in V_n$ is connected by a finite path, \emph{i.e.}, $P_{m,n}>0$.

The incidence matrices defined above correspond to the transpose of the matrices defined at the classical reference in this theory \cite{hps}. This choice, which in our opinion is more mnemotechnical, is done to simplify the reading of the article.

\subsubsection{Ordered Bratteli diagrams and Bratteli-Vershik representations}
An \emph{ordered} Bratteli diagram is a triple \( B=\left( V,E,\preceq \right) \), where \( \left( V,E\right)  \) is a Bratteli diagram and \( \preceq  \) is a partial ordering on \( E \) such that: edges
\( e \) and \( e' \) in $E$ are comparable if and only if $\range(e)=\range(e')$.
This partial ordering naturally defines maximal and minimal edges. Also, 
the partial ordering of $E$ induces another one on paths of $E_{m,n}$ for all $0 \leq m < n$:  
$\left(e_{m+1},\ldots,e_{n}\right) \preceq \left(f_{m+1},\ldots ,f_{n}\right)$ if and only if 
there is $m+1\leq i\leq n$ such that $e_{i}\preceq f_{i}$ and $e_{j}=f_{j}$ for $i<j\leq n$.

Given a strictly increasing sequence of integers 
$\left(n_{k}\right)_{k\geq 0}$ with $n_{0}=0$ one defines the \emph{contraction} or \emph{telescoping} of
$B=\left(V,E,\preceq \right)$ with respect to $\left(n_{k} \right)_{k\geq 0}$ by 
$$\left(\left(V_{n_{k}}\right)_{k\geq
0},\left( E_{n_{k},n_{k+1}}\right)_{k\geq 0},\preceq \right), $$ where $\preceq$ is the order induced in each set of edges 
$E_{n_{k},n_{k+1}}$. The converse operation is called {\it microscoping} (see \cite{hps} and \cite{gps} for more details).
\smallskip

Given an ordered Bratteli diagram \( B=\left( V,E,\preceq \right) \) one defines \( X_{B} \) as the set of infinite paths \( \left(
x_{1},x_{2},\ldots \right)  \) starting in \( v_{0} \) such that $\range(x_n)=\source(x_{n+1})$
for all \( n\geq 1 \). We topologize \( X_{B} \) by postulating a basis of open sets, namely the family of \emph{cylinder sets}
$$
\left[ e_{1},e_{2},\ldots ,e_{n}\right]
=\left\{ \left( x_{1},x_{2},\ldots \right) \in X_{B} \textrm{ } ; \textrm{ }
x_{i}=e_{i},\textrm{ for }1\leq i\leq n\textrm{ }
\right\} .$$

Each \( \left[ e_{1},e_{2},\ldots ,e_{n}\right]  \) is also closed, as is
easily seen, and so \( X_{B} \) is a compact, totally disconnected metrizable space. 
If $(V,E)$ is simple then \( X_{B} \) is Cantor.

When there is a unique point  \( \left( x_{1},x_{2},\ldots \right) \in X_{B} \) such that \( x_{n} \) is (locally) maximal for any \( n\geq 1 \) and a unique
point \( \left( y_{1},y_{2},\ldots \right) \in X_{B} \) such that \( y_{n} \) is (locally) minimal for any \(n \geq 1 \), one says that \( B=\left(
V,E,\preceq \right)  \) is a \emph{properly ordered} Bratteli diagram. We call these particular points \( x_{\mathrm{max}} \) and \(
x_{\mathrm{min}} \) respectively. In this case, we define the map \( V_{B} \) on \( X_{B} \) called the \emph{Vershik map} as follows. Let \( x=\left( x_{1},x_{2},\ldots \right) \in X_{B}\setminus \left\{ x_{\mathrm{max}}\right\}  \) and let \(
n\geq 1 \) be the smallest integer so that \( x_{n} \) is not a maximal edge. Let \( y_{n} \) be the successor of \( x_{n} \) for the corresponding local order and \( \left(
y_{1},\ldots ,y_{n-1}\right)  \) be the unique minimal path in \( E_{0,n-1} \) connecting \( v_{0} \) with the initial vertex of \( y_{n} \). We
set \( V_{B}\left( x\right) =\left( y_{1},\ldots ,y_{n-1},y_{n},x_{n+1},\ldots \right)  \) and \( V_{B}\left( x_{\mathrm{max}}\right)
=x_{\mathrm{min}} \).

The system \( \left( X_{B},V_{B}\right)  \) is called the \emph{Bratteli-Vershik system} generated by \( B=\left( V,E,\preceq \right)
\). The dynamical system induced by any telescoping of \( B \) is topologically conjugate to \( \left( X_{B},V_{B}\right)  \). 

In \cite{hps} it is proved that the system $\left( X_{B},V_{B}\right)$ is minimal whenever
the associated Bratteli diagram $(V,E)$ is simple. Conversely, it is also proved that any minimal Cantor system $\left( X,T\right)$ is topologically conjugate to a Bratteli-Vershik system \( \left(X_{B},V_{B}\right)  \) where $(V,E)$ is simple. 
We say that $B=(V,E,\preceq)$ is a \emph{Bratteli-Vershik representation} of the minimal Cantor system $\left( X,T\right)$ 
if $B$ is properly ordered, $(V,E)$ is simple and $(X,T)$ and $\left( X_{B},V_{B}\right)$ are topologically conjugate. In what follows, each time we consider a representation $B=(V,E,\preceq)$ of $(X,T)$ we will say that $(X,T)$ is given by the Bratteli-Vershik representation $B$ and we will identify $(X,T)$ with $\left( X_{B},V_{B}\right)$.

To have a better understanding of the dynamics of a minimal Cantor system, and in particular 
to understand its group of eigenvalues, one needs to work with a ``good'' Bratteli-Vershik representation. 
So we consider representations such that:

\smallskip

(H1) The entries of $h_{1}$ are all equal to $1$.

(H2) For every $n\geq 2$, $M_{n}>0$.
 
(H3) For every $n\geq 2$, all maximal edges of $E_n$ start in the same vertex of $V_{n-1}$. We assume this vertex is $d_{n-1}$.
\medskip

Classical arguments show that this reduction is possible, in particular (H2) follows from the simplicity of the Bratteli-Vershik representation and (H3) can be deduced from Proposition 2.8 in \cite{hps}. A Bratteli-Vershik representation of a  minimal Cantor system $(X,T)$ satisfying (H1), (H2) and (H3) will be called \emph{proper}. 

\subsubsection{Minimal Cantor systems of finite topological rank}

A minimal Cantor system is of finite (topological)  rank if it admits a Bratteli-Vershik representation such that the number of vertices per level is uniformly bounded by some integer $d$. The minimum possible value of $d$ is called the \emph{topological rank} of the system. We observe that topological and measure theoretical finite rank  notions are different notions. For instance, systems of topological rank one correspond to odometers, whereas in the measure theoretical sense there are rank one systems that are expansive as classical Chacon's example. 

If the minimal Cantor system has finite rank $d$, in the definition of proper representation we will also assume:
\medskip

(H4) For every $n\geq 1$, $d_{n}$ is equal to $d$.
\medskip

This condition can be assumed without loss of generality in the finite rank case. 
Also, to simplify notation and avoid the excessive use of indices, in this last case we will identify $V_{n}$ with $\{1,\ldots, d\}$ for all $n\geq 1$. The level $n$ will be clear from the context. It is not difficult to prove that a  minimal Cantor system of topological finite rank $d$ has a proper representation (see \cite{dfm} for an outline of the proof). 

A minimal Cantor system is \emph{linearly recurrent} if it admits a proper Bratteli-Vershik representation such that the set $\{M_{n};n\geq 2\}$ is
finite. Clearly, linearly recurrent minimal Cantor systems are of finite rank (see \cite{dhs}, \cite{du1}, \cite{du2} and \cite{lr} for more details and properties of this class of systems).

\subsubsection{Kakutani-Rohlin partitions}
Let $B=\left( V,E,\preceq \right)$ be a representation of the minimal Cantor system $(X,T)$. This diagram defines for each $n\geq 0$ a clopen {\it Kakutani-Rohlin} partition of $X$: for $n=0$, 
$\P_{0}=\{B_{0}(v_{0})\}$, where $B_{0}(v_{0})=X$, and for $n\geq 1$    
$$
\P_{n}=\{T^{-j}B_{n}(v); v \in V_n, \ 0 \leq j < h_n(v) \}  ,
$$
where $ B_n(v) = [e_1 , \dots ,e_n]$ and $(e_1 , \dots ,e_n)$ is the unique maximal path from $v_0$ to vertex $v \in V_{n}$. For each $v \in V_n$ the set $\{T^{-j}B_n(v); 0 \leq j < h_n(v) \}$ is called the {\it tower} $v$ of $\P_{n}$. 
It corresponds to the set of all paths from $v_0$ to $v\in V_n$ (there are exactly $h_n(v)$ of such paths). 
Denote by $\T_{n}$ the $\sigma$-algebra generated by the partition $\P_{n}$.
The map $\tau_n: X \to V_n$ is
given by $\tau_n(x)=v$ if $x$ belongs to tower $v$ of $\P_{n}$. The entrance time of 
$x$ to $B_n({\tau_n(x)})$ is given by $r_n(x)=\min\{ j\geq 0; T^jx \in B_n({\tau_n(x)}) \}$.

For each $x=(x_1,x_{2},\ldots) \in X$ and $0 \leq m<n$ define the row vector 
$s_{m,n}(x)\in \NN^{V_{m}}$, called the {\it suffix vector of $x$ between levels $m$ and $n$}, by
$$
s_{m,n}(x,u)=\# \{e \in E_{m,n}; (x_{m+1},\ldots,x_n) \prec e, \source(e)=u \}
$$
at each coordinate $u\in V_m$, where $\prec$ stands for $\preceq$ and $\neq$ simultaneously, and $s_{m,n}(x,u)$ stands for the $u$-th entry of the row vector $s_{m,n}(x)$.
If $y$ is another point in $X$ with $\tau_m(y)=\tau_m(x)$ and $\tau_n(y)=\tau_n(x)$, then it is clear that $s_{m,n}(x)=s_{m,n}(y)$ if and only if $(x_{m+1},\ldots,x_n)=(y_{m+1},\ldots,y_n)$. This fact motivates the following definition. For each $0 \leq m<n$, $u\in V_m$ and $v\in V_n$, define the set
$$
S_{m,n}(u,v)=\set{s_{m,n}(x)\talque x\in X, \ \tau_m(x)=u \y \tau_n(x)=v}.
$$ 
A direct verification shows that the cardinality of $S_{m,n}(u,v)$ is equal to $\matrizp{m}{n}{u}{v}$, {\it i.e.}, the number of paths in $E_{m,n}$ joining $u$ and $v$.
If necessary, to simplify notation we put $s_n(x)=s_{n,n+1}(x)$ and $S_n(u,v)=S_{n,n+1}(u,v)$.
 
A classical computation gives for all $n\geq 1$ (see for example \cite{necesariasuficiente}):
\begin{align}
\label{eq:formulareturn} r_n(x)= s_0(x)+\sum_{i=1}^{n-1} \la s_i(x),h_{1}P_{i} \ra=s_0(x)+\sum_{i=1}^{n-1} \la s_i(x),h_{i} \ra \ ,
\end{align}
where $\la \cdot, \cdot \ra$ is the euclidean inner product. Observe that under  hypothesis (H1), {\it i.e.}, $h_{1}=(1,\ldots,1)$, we have $s_0(x)=0$. Similarly, one can obtain the following general relation between entrance times and suffix vectors of $x \in X$:
\begin{equation}\label{eq:formulareturn_mn}
r_n(x) = r_m(x) + \la s_{m,n}(x),h_m\ra ,
\end{equation}
for $1\leq m<n$. From this equality it follows that for $0\leq \ell < m < n$
\begin{equation}\label{eq:suma_sufijos}
\la s_{\ell,n}(x),h_{\ell}\ra = \la s_{\ell,m}(x),h_{\ell}\ra + \la s_{m,n}(x),h_{m}\ra,
\end{equation}
and particularly
\begin{equation}\label{eq:formulareturn_mn_suma}
\la s_{m,n}(x),h_{m}\ra = \sum_{i=m}^{n-1} \la s_{i}(x),h_{i}\ra.
\end{equation}
Equation \eqref{eq:suma_sufijos} can also be obtained by noticing that for $n\geq 0$ and $x\in X$
\begin{equation}\label{formulasuffix}
s_{n}(x) + s_{n+1}(x)M_{n+1}^{T} = s_{n,n+2}(x),
\end{equation}
and then, for $0\leq \ell < m < n$ we have
\begin{equation}\label{eq:formulasuffix_lmn}
s_{\ell,n}(x) = s_{\ell, m}(x) + s_{m,n}(x)P_{\ell, m}^{T}.
\end{equation}

\subsubsection{Invariant measures}
\label{subsubsec:invmeas}
Let $B=(V,E,\preceq)$ be a Bratteli-Vershik representation of the minimal Cantor system $(X,T)$. Let $\mu$ be an invariant probability measure for this system. 
The measure $\mu$ is determined by the values it gives to $B_n(v)$ for all $n\geq 0$ and $v \in V_n$. Define the column vector
$\mu_{n}=(\mu_n(v) ;  v \in V_{n})$ with $\mu_n(v)=\mu(B_n(v))$. 
A simple computation allows to prove the following useful relation:
\begin{equation}\label{eq:measure}
\mu_{m}=P_{m,n}\mu_{n}
\end{equation}
for all $0\leq m < n$. Also, $\mu(\tau_{n}=v)= \mu\{x\in X;\tau_n(x)=v\}=h_{n}(v) \mu_{n}(v)$ for all $n\geq 1$ and $v\in V_{n}$.

\subsubsection{Clean Bratteli-Vershik representations}\label{subsec:clean}
Let $B=(V,E,\preceq)$ be a proper representation of finite rank $d$ of the minimal Cantor system $(X,T)$.
Recall that in this case we identify $V_{n}$ with $\{1,\ldots, d\}$ for all $n\geq 1$. 
Then, by Theorem 3.3 in \cite{bkms}, there exist a telescoping of the diagram (which keeps the diagram proper) and $\delta_0 >0$ such that:  
\begin{enumerate}
\item For any ergodic measure $\mu$ there exists $I_{\mu} \subseteq \{1,\ldots,d\}$ satisfying: 
\begin{enumerate}
\item $\mu(\tau_{n}=v) \geq \delta_0$ for every $v \in I_{\mu}$ and $n \geq 1$, and   
\item $\lim_{n\to +\infty} \mu(\tau_{n}=v)=0$ for every $v \not \in I_{\mu}$.
\end{enumerate}
\item If $\mu$ and $\nu$ are different ergodic measures then $I_{\mu}\cap I_{\nu}=\emptyset$.  
\end{enumerate}
When a proper Bratteli-Vershik representation of finite rank $d$ satisfies the previous properties we say it is \emph{clean}. We remark that this is a modified version of the notion of \emph{clean} Bratteli diagram given in \cite{rangofinito} that is inspired by the results of \cite{bkms}. This property will be  
relevant for formulating our main result in the measurable case. 

In \cite{bkms}, systems such that $I_{\mu}=\{1,\ldots,d\}$ for some ergodic measure 
$\mu$ are called of \emph{exact finite rank}. Clearly, those systems are uniquely ergodic.

\section{Continuous eigenvalues of minimal Cantor systems}
\label{sec:continuous}

In this section we show a necessary and sufficient condition for a complex number to be a continuous eigenvalue of a minimal Cantor system. The condition is given in terms of the combinatorial objects associated to a proper Bratteli-Vershik representation of the system. The proof follows the lines and some ideas developed to prove a general necessary condition in \cite{rangofinito}. 
 
\subsection{The necessary and sufficient condition}
Let $(X,T)$ be a minimal Cantor system given
by a proper Bratteli-Vershik representation $B=\left( V,E,\preceq \right)$. The most general necessary and sufficient condition for a complex number $\lambda=\exp(2i\pi\alpha)$ to be a continuous eigenvalue of $(X,T)$ states that the map
$\lambda^{r_n(\cdot)}$ converges uniformly (see Proposition 12 in \cite{necesariasuficiente}). In order to achieve the uniform convergence, several simpler necessary conditions relying on the combinatorics of the Bratteli-Vershik representation $B$ have been proposed. 
We recall the necessary condition proved in {\cite{rangofinito}} that serves as motivation to the main result of this section. Denote by $ \vvert\cdot  \vvert$ the distance to the nearest integer vector.

\begin{theo}
\label{th:necessary-continuous} 
Let $(X,T)$ be a minimal Cantor system given
by a proper Bratteli-Vershik representation $B=\left( V,E,\preceq \right)$. 
If $\lambda=\exp(2i\pi \alpha)$ is a continuous
eigenvalue of $(X,T)$ then
$$
\displaystyle \sum_{n\ge 1} \vvert \alpha h_1 P_n \vvert = \sum_{n\ge 1} \vvert \alpha h_n \vvert <+\infty \ . 
$$
\end{theo}

Let $\lambda=\exp(2i\pi\alpha)$ be a continuous eigenvalue of $(X,T)$ 
as in Theorem \ref{th:necessary-continuous}. Then, for all $n\geq 1$ there exist a real vector $\eta_n$ and an integer vector $\nu_n$ such that
\begin{equation}
\label{eq:eta}
\alpha h_1 P_n = \alpha h_n =\eta_n + \nu_n \hbox{ and } \eta_n \xrightarrow[n\to +\infty]{}0 \ .
\end{equation}
Moreover, Theorem 5 in \cite{rangofinito} states that such a decomposition satisfies that for all large enough $n$ 
\begin{equation}
\label{eq:eta1}
\eta_{n+1}=\eta_n M_{n+1} \text{ and } \nu_{n+1}=\nu_n M_{n+1} \ .  
\end{equation}
A classical computation allows us to deduce the possible values for $\alpha$ from these two conditions. This is part (2) of Corollary 7 in \cite{rangofinito} but we give a proof for completeness. Fix large integers $1 \leq m <  n$ such that \eqref{eq:eta1} holds for such values 
and multiply (the row vector) $\alpha h_m$ by (the column vector) $\mu_m$, where $\mu$ is any invariant probability measure. From \eqref{eq:eta}, \eqref{eq:measure} and \eqref{eq:eta1}  we get,
\begin{align*}
\alpha &=\alpha h_m \cdot \mu_m = \eta_m \cdot \mu_m  +  \nu_m \cdot \mu_m
= \eta_m \cdot P_{m,n} \cdot \mu_{n} + \nu_m \cdot \mu_m = \eta_n \cdot \mu_{n} + \nu_m \cdot \mu_m, 
\end{align*}
where in the first equality we have used the fact that 
$\displaystyle h_m \cdot \mu_m=1$.
Taking $n\to +\infty$ and using \eqref{eq:eta} we get that 
\begin{equation}
\label{eq:eta2}
\alpha= \nu_m \cdot \mu_m \text{ and } \eta_m \cdot \mu_m=0 \text{ for every large enough $m\in \NN$}. 
\end{equation}

We stress the fact that many values of $\lambda=\exp(2i\pi \alpha)$ with $\alpha$ as in \eqref{eq:eta2} could not be continuous eigenvalues of $(X,T)$. This fact strongly relies on the local orders of the Bratteli-Vershik representation $B$.

The general necessary and sufficient condition we present below refines the one in Theorem \ref{th:necessary-continuous} incorporating the local orders of the Bratteli-Vershik representation of the minimal Cantor system. This is achieved by considering the suffix vectors  defined at each level of the diagram. While submitting this article we remarked the similarity of this result with Theorem 4.1 in \cite{sadunpriebe}, where the authors provide a necessary and sufficient condition to be a continuous eigenvalue for a special class of fusion tilings. 

\begin{theo}
\label{theo:nec-suff-cont} Let $(X,T)$ be a minimal Cantor system given
by a proper Bratteli-Vershik representation $B=\left( V,E,\preceq \right)$. 
The following conditions are equivalent,
\begin{enumerate}
\item
$\lambda=\exp(2i\pi \alpha)$ is a continuous eigenvalue of $(X,T)$;
\item 
$\displaystyle \sum_{n\ge 1}  \max_{x\in X}  \vvert \alpha \langle  s_n(x) , h_n  \rangle\vvert  <+\infty
 $;
\item 
$\displaystyle \sum_{n\ge 1}  \max_{s \in S_n(u_n,u_{n+1})}  \vvert  \alpha \langle  s , h_n  \rangle\vvert  <+\infty
 $ for any sequence of vertices $(u_{n}; n\geq 1)$ with $u_n \in V_n$.
\end{enumerate}
\end{theo}
\medskip

Assume that $(X,T)$ is a minimal Cantor system given
by a proper Bratteli-Vershik representation $B=\left( V,E,\preceq \right)$ as in Theorem \ref{theo:nec-suff-cont}. 
We will need two preliminary lemmas to prove that (1) is equivalent with (2) and (3). The first one is almost identical to Lemma 4 in \cite{rangofinito} so we omit its proof. We only remark that the proof is a simple use of properties (H2) and (H3) in the definition of proper Bratteli-Vershik representation. 

\begin{lemma}
\label{lemma:twopoints} 
Let $(j_n; n \geq 1)$ be a sequence of positive integers such that  $j_{n+1} -
j_{n} \geq 3$ and let $(e_{j_n+1}; n\geq 1)$ be a sequence of edges of the Bratteli diagram with $e_{j_n+1} \in E_{j_n+1}$.
Then, there exist points $x=(x_1,x_2,\ldots)$ and $y=(y_1,y_2,\ldots)$ in $X$ such that for all $\ n\geq 1$,
\begin{enumerate}
\item
$x_{j_n+1}=e_{j_n+1}$, $\range(x_{j_n+1})=\range(y_{j_n+1})$ and $s_{j_n}(y) =(0,\ldots,0)$ (\emph{i.e.}, $y_{j_n+1}$ is a maximal edge);
\item
$x_j = y_j$ for $j_n+2\leq j \leq j_{n+1}-1$; 
\item
$s_{j_{n+1}-1} (x) = s_{j_{n+1}-1} (y) = (0, \ldots , 0)$ (\emph{i.e.}, $x_{j_{n+1}}$ and $y_{j_{n+1}}$ are maximal edges).
\end{enumerate}
\end{lemma}
\medskip

\begin{lemma}
\label{prop:tendtozero} 
Let $\lambda=\exp(2i\pi \alpha)$ be a continuous eigenvalue of $(X,T)$. 
For every $n\geq 1$ let $\eta_n$ and $\nu_n$ be the
real and integer vectors satisfying properties \eqref{eq:eta} and \eqref{eq:eta1}. Then, 
$\displaystyle
\max_{x\in X} |\langle  s_n(x) , \eta_n \rangle | \to_{n\to +\infty} 0. 
$
\end{lemma}
\begin{proof}
Since $\lambda$ is a continuous eigenvalue we have that the sequence of maps $(\vvert \alpha r_n(\cdot) \vvert;n\geq 1)$ converges uniformly (Proposition 12 in \cite{necesariasuficiente}).

Fix $0< \varepsilon < 1/8$. By property \eqref{eq:eta}, equality \eqref{eq:formulareturn} and the uniform convergence of $(\vvert \alpha r_n(\cdot) \vvert;n\geq 1)$, there exists $n_0\geq 1$ such that for all $n\geq n_0$ and  
$x\in X$, $\Vert \eta_n \Vert < \varepsilon  < 1/8$ and  
$$
\vvert \langle s_n (x) , \eta_n \rangle \vvert
= 
\vvert \langle s_n (x) , \eta_n + \nu_n \rangle \vvert
=
\vvert \langle s_n (x) , \alpha h_n \rangle \vvert
=
\vvert \alpha(r_{n+1}(x)-r_n(x)) \vvert
< \varepsilon .
$$

Write $\langle s_n (x) , \eta_n \rangle = \varepsilon_n (x) + E_n(x)$ with $|\varepsilon_n (x)| < \varepsilon$ and $E_n(x)$ an integer (the closest one). 
Notice that the sequence of maps $(\varepsilon_n(\cdot);n\geq 1)$ converges uniformly to $0$.

For $n\geq n_0$ consider the set $A_n = \{ x \in X  ; E_n(x) = 0 \}$.
Observe that this set is not empty (consider a point with a maximal edge at level $n+1$) and closed
(the map $x \mapsto \langle s_n (x) , \eta_n \rangle$ is locally constant). Let us check that it is $T$-invariant. 
Take $x \in A_n$. We have to consider three cases: $s_n (x) = s_n (Tx)$, $s_n (x) = 0$ with  $s_n (x) \not = s_n (Tx)$, and 
$s_n (x) = s_n (Tx) + e$ for some vector $e$ from the canonical base. In the first case it is obvious that $Tx  \in A_n$ when $x \in A_n$. 

In the second one $(x_1,\ldots,x_{n+1})$ is formed by maximal edges and thus $(Tx)_{n+1}$ is a minimal edge. Therefore, 
$s_n(Tx)$ is the $v$-th column of $M_{n+1}$ minus the $u$-th canonical vector, where $u=\source((Tx)_{n+1})$ and $v=\range((Tx)_{n+1})$. Then, 
$$|\langle s_n(Tx), \eta_n \rangle |= \left | \sum_{u'\in V_n} \eta_n(u') M_{n+1}(u',v) - \eta_n(u) \right| = |  \eta_{n+1}(v) - \eta_n(u)| \leq \frac{1}{4} \ ,$$
where in the second equality we have used the relation in \eqref{eq:eta1}. Hence, $E_n(Tx) = E_n(x) = 0$ and $Tx \in A_n$.

In the last case, 
\begin{align*}
|E_n(Tx ) - E_n(x)| = & |\varepsilon_n (x) -\varepsilon_n (Tx) + \langle s_n (T x) ,  \eta_n \rangle - \langle s_n (x) ,  \eta_n \rangle |\\
 \leq  & \frac 14 + |\langle e ,  \eta_n \rangle | \leq \frac 14 +  \Vert\eta_n\Vert < \frac 12 .
\end{align*}
Therefore, $E_n(Tx) = E_n(x) = 0$ and $Tx \in A_n$.

By minimality, we obtain that $A_n=X$. This implies that for all $n \geq n_0$ 
$$
\vvert \langle  s_n(x) , \eta_n \rangle \vvert  = | \langle  s_n(x) ,\eta_n \rangle | = | \varepsilon_n (x) | \ ,
$$
which achieves the proof.
\end{proof}

\begin{proof}[Proof of Theorem \ref{theo:nec-suff-cont}]
First we prove that (2) and (3) are equivalent. Clearly, the series in (2) is an upper bound of the series in (3), so (2) implies (3). Now, it is not difficult to prove that there exist sequences $(u_n;{n\geq 1})$ and 
$(v_n;{n\geq 1})$, with $u_n,v_n \in V_n$, such that: 
$\displaystyle \max_{s \in S_n(u_n,u_{n+1})} \vvert \alpha \langle  s , h_n  \rangle\vvert
=\max_{x \in X} \vvert \alpha \langle  s_n(x) , h_n  \rangle\vvert$ if $n$ is odd; 
and 
$\displaystyle \max_{s \in S_n(v_n,v_{n+1})} \vvert \alpha \langle  s , h_n  \rangle\vvert
=\max_{x\in X} \vvert \alpha \langle  s_n(x) , h_n  \rangle\vvert$ if $n$ is even. 
The sum of the series 
$\displaystyle \sum_{n\ge 1}  \max_{s \in S_n(u_n,u_{n+1})}  \vvert  \alpha \langle  s , h_n  \rangle\vvert + \displaystyle \sum_{n\ge 1}  \max_{s \in S_n(v_n,v_{n+1})}  \vvert  \alpha \langle  s , h_n  \rangle\vvert$ is an upper bound of series (2), thus (3) implies (2). 

Now we prove that (1) implies (2) (and thus (3)). Let $\lambda=\exp(2i\pi \alpha)$ be a continuous eigenvalue of
$(X,T)$. Then, there exist a real vector 
$\eta_n$ and an integer vector $\nu_n$ satisfying conditions \eqref{eq:eta} and \eqref{eq:eta1}.
In particular, 
$$
\displaystyle
\alpha h_n = \eta_n + \nu_n \hbox{ and } \eta_n \xrightarrow[n\to +\infty]{}
0,
$$ 
for all $n\geq 1$.

Thus, to get condition (2) of the theorem it is enough to prove that the series $\displaystyle\sum_{n\geq 1} \max_{x\in X} | \langle s_n(x) , \eta_n \rangle |$ converges.

For $n\geq 1$ let $z^{(n)}=(z^{(n)}_1,z^{(n)}_2,\ldots) \in X$ be such that 
$$
| \langle  s_n(z^{(n)}) , \eta_n  \rangle |= \max_{x\in X}| \langle  s_n(x) , \eta_n  \rangle |  \ . 
$$
We set $e_{n+1} = z^{(n)}_{n+1}$ and ${\mathfrak s}_n = s_n(z^{(n)})$. So, it suffices to prove the following convergence,
\begin{align}
\label{series:suffisant}
\sum_{n\geq 1} | \langle {\mathfrak s}_n , \eta_n \rangle | < +\infty .
\end{align}
Let
$
I^+=\{n \geq 1; \langle {\mathfrak s}_{n}, \eta_n\rangle   \geq 0 \}, \ I^-=\{n \geq 1 ;
\langle {\mathfrak s}_{n}, \eta_n\rangle  <0\}.
$
To prove \eqref{series:suffisant}  we only need
to show that
$$
\sum_{n\in I^+} \langle {\mathfrak s}_{n},\eta_n\rangle   < +\infty \; \; \textrm{and}\; \;  -
\sum_{n\in I^-} \langle {\mathfrak s}_{n},\eta_n\rangle   < +\infty.
$$
Since the arguments in both cases are similar we only
prove the first one. Moreover, to prove $\sum_{n\in I^+}
\langle {\mathfrak s}_{n}, \eta_n\rangle  < +\infty$ we only show 
$
\sum_{n\in I^+\cap 3\NN}
\langle {\mathfrak s}_n,\eta_n\rangle   < +\infty.
$ 
In a similar way one proves the convergence of series 
$\sum_{n\in I^+\cap (3\NN+1)} \langle {\mathfrak s}_n,\eta_n\rangle$ and $\sum_{n\in I^+\cap (3\NN+2)} \langle {\mathfrak s}_n,\eta_n\rangle$.

Assume $I^+ \cap 3\NN$ is infinite, if not the result follows
directly. Order its elements: $1 <j_1 < j_2< \ldots <j_n<\ldots$.
From Lemma \ref{lemma:twopoints} there exist two points $x= (x_1,x_2, \ldots)$ and $y =(y_1,y_2,\ldots)$ in $X$ such that for all  $n\geq 1$, 

\begin{enumerate}
\item
$x_{j_n+1}=e_{j_n+1}$, $s_{j_n}(y) = (0,\ldots,0)$ and $\range(x_{j_n+1})=\range(y_{j_n+1})$;
\item $x_j = y_j$ for $j_n+2\leq j \leq j_{n+1}-1$; 
\item
$s_{j_{n+1}-1} (x) = s_{j_{n+1}-1} (y) = (0, \ldots , 0)$.
\end{enumerate}
We also set $x_1=y_1,\ldots,x_{j_1-1}=y_{j_1-1}$ and $s(y_{j_1})=(0,\ldots,0)$. 
Hence, we have 
\begin{enumerate}
\item
$s_{j_n}(x) =  {\mathfrak s}_{j_n}$ and $ s_{j_n}(y)  = (0,\ldots,0)$;
\item
$s_{j-1}(x) - s_{j-1}(y) =  (0,\ldots,0)$ for $j_n+2\leq j \leq j_{n+1}-1$. 
\end{enumerate}

Now, from the definition of the return function in \eqref{eq:formulareturn} and properties of points $x$ and $y$ just constructed we get for all $m>1$, 

\begin{align*}
\alpha (r_m (x) - r_m (y)) 
&= \alpha 
\sum_{n\in \{1,\ldots,m-1\}} \langle s_n(x)-s_n(y), h_n\rangle  
\\
&= \alpha 
\sum_{n\in \{1,\ldots,m-1\}\cap I^+\cap 3\NN} \langle s_n(x)-s_n(y), h_n\rangle  
\\
&=\alpha \sum_{n\in \{1,\ldots,m-1\}\cap I^+\cap 3\NN} \langle {\mathfrak s}_{n}, h_n \rangle 
\\
&=\sum_{n\in \{1,\ldots,m-1\}\cap I^+\cap 3\NN} \langle {\mathfrak s}_{n},\eta_{n}+\nu_{n}\rangle \ .
\end{align*}

From Proposition 12 in \cite{necesariasuficiente} we have $\alpha (r_m (x) - r_m (y))$ converges~$\mod \ZZ$ when $m \to +\infty$. Then, $\sum_{n\in \{1,\ldots,m-1\}\cap I^+\cap 3\NN} \langle {\mathfrak s}_{n}, \eta_n\rangle $ converges~$\mod \ZZ$ when $m\to +\infty$ too. 
But, from Lemma \ref{prop:tendtozero}, $\langle {\mathfrak s}_{n},\eta_n\rangle $ tends to 0 when $n\to +\infty$, hence the series $\sum_{n\in I^+\cap 3\NN} \langle {\mathfrak s}_{n},\eta_n\rangle $ converges. This proves that (1) implies (2). 
\medskip

Now we prove that (2) implies (1). Assume that 
$\sum_{n\ge 1}  \max_{x\in X} \vvert \alpha  \langle  s_n(x) , h_n  \rangle  \vvert$ converges and let us prove that
$\lambda=\exp(2i\pi \alpha)$ is a continuous eigenvalue of $(X,T)$.

By Proposition 12 in \cite{necesariasuficiente}, it suffices to show that 
$(\alpha r_n(x);n\in \NN)$ converges~$\mod  \mathbb{Z}$ uniformly in $x$.  
For $1 \leq m< n$ and $x \in X$ we have,
\begin{align*}
\left| \vvert \alpha r_n(x) \vvert - \vvert \alpha r_m(x) \vvert  \right| \leq & \vvert \alpha (r_n(x) - r_m (x)) \vvert \\
 = & \vvert \alpha \left( \sum_{k=m}^{n-1} \langle s_k(x) , h_k \rangle  \right) \vvert \\
 \leq &  \displaystyle \sum_{k=m}^{n-1}  \max_{y\in X} \vvert \langle  s_k(y) ,  \alpha h_k \rangle \vvert \ .
\end{align*}
Then, condition (2) implies that $\vvert \alpha r_n \vvert$ is a Cauchy sequence for the uniform convergence. This achieves the proof.
\end{proof}

In proving that (2) implies (1) we used Proposition 12 in \cite{necesariasuficiente} and the definition of the map $r_n$. It is worth pointing out that the implication of Proposition 12 that we used does not need the diagram to be proper, it is enough for it to be only a representation, {\it i.e.}, being properly ordered and simple. We state this fact as a corollary due to its relevance in examples and applications where the incidence matrices of the corresponding Bratteli-Vershik representation are not necessarily strictly positive.

\begin{coro}\label{rmrk:cond_suf_sin_simple}
Let $(X,T)$ be a minimal Cantor system given
by a Bratteli-Vershik representation $B=\left( V,E,\preceq \right)$. 
If the real number $\alpha$ satisfies condition (2) of Theorem \ref{theo:nec-suff-cont}, then $\lambda=\exp(2i\pi \alpha)$ is a continuous eigenvalue of the system.
\end{coro}

As stated before, conditions \eqref{eq:eta} and \eqref{eq:eta1} allow us to compute all possible values of $\alpha$ such that $\lambda=\exp(2i\pi \alpha)$ is a candidate to be a continuous eigenvalue. The main problem is to know whether they really correspond to continuous eigenvalues. This is related to the local orders of the Bratteli-Vershik representations and this is the point where Theorem \ref{theo:nec-suff-cont} plays a role. The following corollary (that is in the folklore) shows that those candidates that are roots of unity are always continuous eigenvalues. 

\begin{coro}\label{coro:rootsofunity}
Let $(X,T)$ be a minimal Cantor system given by a proper Bratteli-Vershik representation $B=(V,E,\preceq)$. Then, $\lambda=\exp(2i\pi p/q)$ is a continuous eigenvalue of $(X,T)$ if and only if the rational number $p/q$ satisfies \eqref{eq:eta}. Equivalently, if and only if $q$ divides the coordinates of the vector of heights $h_n$ for every large enough $n\geq 2$.
\end{coro}
\begin{proof}
As discussed before all continuous eigenvalues satisfy \eqref{eq:eta}. Conversely, 
if $\alpha=p/q$ satisfies \eqref{eq:eta} then necessarily $q$ divides the coordinates of the height vector $h_n$ and $\eta_n=0$ for all large enough $n \in \NN$. Hence, as $\vvert \alpha \langle s_n(x), h_n\rangle \vvert = \vvert \langle s_n(x), \eta_n\rangle \vvert$ for every $n\geq 1$, the sum in condition (2) of Theorem \ref{theo:nec-suff-cont} is finite and $\lambda=\exp(2i\pi p/q)$ is a continuous eigenvalue of $(X,T)$. 
\end{proof}

The case of irrational continuous eigenvalues ({\it i.e.}, continuous eigenvalues that are not roots of unity) is more involved. In Section \ref{examples} we make a slightly more in depth analysis related to this kind of eigenvalues.

We also stress that computations related to the conditions of Theorem \ref{theo:nec-suff-cont} can be complicated as they might require a lot of information about the Bratteli-Vershik representation of a system. However, in the case of linearly recurrent minimal Cantor systems, since their suffix vectors are uniformly bounded, condition (2) of Theorem \ref{theo:nec-suff-cont} can be reduced to 
$\displaystyle \sum_{n\ge 1} \vvert \alpha  h_n \vvert  <+\infty$ as it was already shown in \cite{necesariasuficiente}. Therefore, in this case, the unique significant combinatorial data is the collection of incidence matrices of the Bratteli-Vershik representation. Unfortunately, many relevant examples of Cantor minimal systems are not linearly recurrent, so the local orders of their Bratteli-Vershik representation, or equivalently the suffix vectors at each level, cannot be neglected.  

\section{Measurable eigenvalues of finite rank minimal Cantor systems}
\label{sec:measurable}

Let us recall an abstract necessary and sufficient condition for a complex number to be a measurable eigenvalue of a minimal Cantor system.  

\begin{theo}[\cite{necesariasuficiente}, Theorem 7]\label{theo:cns_old_rho}
Let $(X,T)$ be a minimal Cantor system given by a proper  Bratteli-Vershik representation $B=(V,E,\preceq)$. Let  $\mu$ be an ergodic probability measure. Then, $\lambda=\exp(2i \pi \alpha)$ is an eigenvalue of $(X,T)$  for $\mu$ if and only if there exists a sequence of functions
$(\rho_n:V_n\to \RR \talque n\geq 1)$
such that a subsequence of $\left(\exp(2i \pi\alpha(r_n+\rho_n\circ\tau_n)) \talque n\geq 1\right)$ converges $\mu$-almost everywhere in $X$.
\end{theo}
A main issue in last theorem is the construction of functions $\rho_n$. 
In this section we obtain a new necessary and sufficient condition for a complex number to be a measurable eigenvalue of a finite rank minimal Cantor system that does not depend on the existence of the functions 
$\rho_n$. In addition, this condition gives an idea of ââhow such maps can be constructed (Theorem \ref{theo:cns_sinrho}). It is only based on the combinatorial structure of Brattelli-Vershik representations of finite rank minimal Cantor systems.  

Another necessary and sufficient condition valid for systems of exact rank ({\it i.e.}, $I_\mu=\{1,\ldots,d\}$) is presented in Theorem \ref{theo:cnsvpmedible}. It is formulated as a convergence of a series, but again its terms depend on the existence of auxiliary functions $\rho_n$ that we do not control. We include this condition since it follows previous work on the subject in the linearly recurrent case. 

\subsection{Necessary and sufficient condition controlled by the local orderings of the Bratteli-Vershik representation}

Let $(X,T)$ be a minimal Cantor system given by a proper and clean Bratteli-Vershik representation $B=(V,E,\preceq)$ of finite rank $d$ and let $\mu$ be an ergodic probability measure. We start with a classical analysis of an eigenfunction $f\in L^{2}(X,\mu)$, with $\abs{f}=1$, associated to some eigenvalue 
$\lambda$. Let $n\geq 1$. Recall that $\T_{n}$ is the $\sigma$-algebra generated by the partition 
$\P_{n}=\{T^{-j}B_{n}(u); u \in V_n, \ 0 \leq j < h_n(u) \}$. We have
\begin{align*}
\EE(f|\T_n) &= \sum_{u\in V_n}\sum_{j=0}^{h_n(u)-1}{\bf 1}_{T^{-j}B_n(u)}\frac{1}{\mu_n(u)}\int_{T^{-j}B_n(u)}fd\mu \\
 &= \sum_{u\in V_n}\sum_{j=0}^{h_n(u)-1}{\bf 1}_{T^{-j}B_n(u)}\frac{1}{\mu_n(u)}\int_{B_n(u)}f\circ T^{-j}d\mu \\
 &= \sum_{u\in V_n}\sum_{j=0}^{h_n(u)-1}{\bf 1}_{T^{-j}B_n(u)}\frac{1}{\mu_n(u)}\int_{B_n(u)}\lambda^{-j}fd\mu.
\end{align*}

We define for each $n\geq 1$ and $u\in V_n$ the real numbers $c_n(u)$ and 
$\rho_n(u)$ in $[0,1)$ by 
\begin{equation}\label{eq:definicion_c_ro}
c_n(u)\lambda^{-\rho_n(u)}=\frac{1}{\mu_n(u)}\int_{B_n(u)}f d\mu  .
\end{equation}
Then we can write $\ds \EE(f|\T_n)(x) = c_n(\tau_n(x))\lambda^{-r_n(x)-\rho_n(\tau_n(x))}$, where we recall $r_n(x)$ is the entrance time of $x$ to $B_n(\tau_n(x))$.
\smallskip 

We have the following known property.
\begin{lemma}[\cite{rangofinito}, Lemma 17]\label{lemma:convergencia_c}
For each vertex $u \in\set{1,\ldots, d}$ 
$$
\mu(\tau_n=u)(1-c_n(u))\xrightarrow[n\to+\infty]{}0,
$$
and therefore for each $u\in I_{\mu}$, $c_n(u)\xrightarrow[n\to+\infty]{}1$.
\end{lemma}
Notice that we have identified $V_n$ with $\set{1,\ldots, d}$ for each $n\geq 1$.
\medskip

The following lemma will be useful to better understand our main result.

We say that a sequence of real numbers $(a_{m,n}; m,n \geq 1)$ converges to $a$ when $m\to +\infty$ {\it uniformly for $n>m$}, if for any $\varepsilon >0$ there exists $m_0\geq 1$ such that  for any $n> m\geq m_0$ we have $|a_{m,n} - a|\leq \varepsilon$.

\medskip

\begin{lemma}\label{lemma:comportamiento_cuocientes}
Let $(X,T)$ be a minimal Cantor system given by a proper and clean Bratteli-Vershik representation $B=(V,E,\preceq)$ of finite rank $d$ and let $\mu$ be an ergodic probability measure of the system. Then,
\begin{itemize}
\item[(1)] For $u\not\in I_{\mu}$ and $v\in I_{\mu}$
$$
\frac{h_m(u)}{h_n(v)}P_{m,n}(u,v)\xrightarrow[m\to+\infty]{}0
$$
uniformly for $n>m$. 
\item[(2)] For each $m\geq 1$, $u\in V_m$ and $v\in I_{\mu}$
$$
\frac{h_m(u)}{h_n(v)}P_{m,n}(u,v)\xrightarrow[n\to+\infty]{} \mu(\tau_m=u).
$$
\end{itemize}
\end{lemma}

\begin{proof}
(1) Recall $\delta_0 >0$ is the constant appearing in the definition of clean diagram, and take $u\not\in I_{\mu}$ and $v\in I_{\mu}$ (so $\mu(\tau_n=v) \geq \delta_0$ for every $n\geq 1$). To prove the statement it suffices to notice, using $\mu_m=P_{m,n}\mu_n$ for every $n>m\geq 1$, that the following inequality holds
\begin{align*}
\mu(\tau_m=u) &\geq \frac{h_m(u)}{h_n(v)}P_{m,n}(u,v)\mu(\tau_n=v) \geq \delta_0\frac{h_m(u)}{h_n(v)}P_{m,n}(u,v),
\end{align*}
and to make $m\to+\infty$.

Statement (2) can be proved using the same ideas in \cite[Lemma 11]{dfm}.
Recall that the measure of the set $B_m(u)$ is denoted by $\mu_m(u)$. Set $m\geq 1$, $u\in V_m$, $v\in I_\mu$ and $0< \varepsilon < \delta_0$. 
For $\mu$-almost every $x\in X$, the pointwise ergodic theorem and Egorov's theorem give us a set $A$ with $\mu(A)>1-\varepsilon$ and a positive integer $N_0$ such that for all $x\in A$ and $N\geq N_0$
\begin{equation}\label{eq:ergodico_egorov}
\abs{\frac{1}{N}\sum_{k=0}^{N-1} {\bf 1}_{B_m(u)}(T^{k}x) - \mu_m(u) }< \varepsilon.
\end{equation}

Take $n>m$ such that $h_n(v)>N_0$. We can find $j$, with $0\leq j\leq \left\lfloor\frac{\varepsilon h_n(v)}{\delta_0}\right\rfloor$, such that $A\cap T^{-h_n(v)-j+1}B_n(v)\neq\emptyset$. Indeed, since $\mu$ is invariant and $v\in I_{\mu}$ we have 
\begin{align*}
 &\mu\left(\bigcup_{j=0}^{\lfloor\varepsilon h_n(v)/\delta_0\rfloor}T^{-h_n(v)-j+1}B_n(v)\right) 
 = \left(\left\lfloor\frac{\varepsilon h_n(v)}{\delta_0}\right\rfloor + 1\right)\mu_n(v) > \frac{\varepsilon}{\delta_0}\mu(\tau_n=v)\geq \varepsilon.
\end{align*} 
\smallskip

Now, taking $x^{(n)}\in A\cap T^{-h_n(v)-j+1}B_n(v)$, relation \eqref{eq:ergodico_egorov} implies that
\begin{equation}\label{eq:ergodico_egorov_hn}
\abs{\frac{1}{h_n(v)+j}\!\!\!\!\sum_{k=0}^{h_n(v)+j-1}\!\!\!\!\! {\bf 1}_{B_m(u)}(T^{k}x^{(n)}) - \mu_m(u) } < \varepsilon.
\end{equation}
Let us write
\begin{align*}
&\frac{1}{h_n(v)+j}\!\!\!\!\sum_{k=0}^{h_n(v)+j-1}\!\!\!\!\!{\bf 1}_{B_m(u)}(T^{k}x^{(n)}) \\
= & \frac{1}{h_n(v)+j}\sum_{k=0}^{j-1}{\bf 1}_{B_m(u)}(T^{k}x^{(n)})  + \frac{1}{h_n(v)}\!\!\!\!\sum_{k=j}^{h_n(v)+j-1}\!\!\!\!\!{\bf 1}_{B_m(u)}(T^{k}x^{(n)}) \\
 &- \frac{j}{h_n(v)(h_n(v)+j)}\!\!\!\!\sum_{k=j}^{h_n(v)+j-1}\!\!\!\!\!{\bf 1}_{B_m(u)}(T^{k}x^{(n)}).
\end{align*}
Notice that the modulus of the first and third terms on the right side are each bounded by
$$
\frac{j}{h_n(v)+j} < \frac{\varepsilon}{\delta_0}.
$$
Combining this with \eqref{eq:ergodico_egorov_hn} we obtain
\begin{equation}\label{eq:ergodico_egorov_final}
\abs{\frac{1}{h_n(v)}\!\!\!\!\sum_{k=j}^{h_n(v)+j-1}\!\!\!\!\!{\bf 1}_{B_m(u)}(T^{k}x^{(n)}) - \mu_m(u)} < \varepsilon\!\left(1+\frac{2}{\delta_0}\right).
\end{equation}
If we define $y^{(n)}=T^{h_n(v)+j-1}x^{(n)}\in B_n(v)$ (notice that $y^{(n)}$ depends on $n$ and $v$), relation \eqref{eq:ergodico_egorov_final} leads to
\begin{equation}\label{eq:ergodico_egorov_convergencia}
\frac{P_{m,n}(u,v)}{h_n(v)}=\frac{1}{h_n(v)}\!\!\sum_{k=0}^{h_n(v)-1}\!\!\!{\bf 1}_{B_m(u)}(T^{-k}y^{(n)})
\xrightarrow[n\to+\infty]{} \mu_m(u).
\end{equation}
Multiplying both sides of \eqref{eq:ergodico_egorov_convergencia} by $h_m(u)$ gives statement (2). 
\end{proof}

We are ready to state the main result of the section. 

\begin{theo}
\label{theo:cns_sinrho}
Let $(X,T)$ be a minimal Cantor system given by a proper and clean Bratteli-Vershik representation $B=(V,E,\preceq)$ of finite rank $d$. Let  $\mu$ be an ergodic probability measure. 
Then, $\lambda = \exp(2 i \pi \alpha)$ is an eigenvalue of $(X,T)$ for $\mu$
if and only if one of the following two equivalent conditions hold:
\begin{itemize}
\item[(1)] For all $v\in I_{\mu}$,
$$
\sum_{u\in I_{\mu}}\frac{h_m(u)}{h_n(v)}\abs{\sum_{s\in \sufijocbb{m}{n}{u}{v}}\lambda^{\la s, h_m\ra}}\xrightarrow[m\to+\infty]{} 1
$$
uniformly for $n>m$.
\vskip 0.5cm

\item[(2)] For all $u\in\set{1,\ldots,d}$ and $v\in I_{\mu}$,
$$
\frac{h_m(u)}{h_n(v)}\left[\matrizp{m}{n}{u}{v}-\abs{\sum_{s\in \sufijocbb{m}{n}{u}{v}}\lambda^{\la s, h_m\ra}}\right]\xrightarrow[m\to+\infty]{} 0
$$
uniformly for $n>m$
\end{itemize}
\end{theo}
The proof of the theorem has been divided into three parts.

\subsubsection{Proof that (1) and (2) are equivalent}
First, (2) implies (1) follows from $\sum_{u\in V_m} \frac{h_m(u)}{h_n(v)} \matrizp{m}{n}{u}{v}=1$ and part (1) of Lemma \ref{lemma:comportamiento_cuocientes}. 

To prove that (1) implies (2) we proceed by contradiction. Consider $\varepsilon >0$ 
and use (1) to get $m_0 \geq 1$ such that for all $n > m\geq m_0$ and all $v \in I_{\mu}$ 
\begin{align}\label{eq:blabla}
\sum_{u \in I_{\mu}} \frac{h_m(u)}{h_n(v)}\abs{\sum_{s\in \sufijocbb{m}{n}{u}{v}}\lambda^{\la s, h_m\ra}} > 1-\varepsilon \ . 
\end{align}
Now, assume that for some large $n>m$ and for some $u_0 \in V_m$ and $v_0 \in I_\mu$
we have
\begin{align}\label{eq:blabla1}
\frac{h_m(u_0)}{h_n(v_0)}\left[\matrizp{m}{n}{u_0}{v_0}-\abs{\sum_{s\in \sufijocbb{m}{n}{u_0}{v_0}}\lambda^{\la s, h_m\ra}}\right] \geq \varepsilon \ . 
\end{align}
Then, 
{\small
\begin{align*}
1 &= \sum_{u \in V_m} \frac{h_m(u)}{h_n(v_0)} \matrizp{m}{n}{u}{v_0} \\
  &= \sum_{u \in V_m} \frac{h_m(u)}{h_n(v_0)} \left (\matrizp{m}{n}{u}{v_0} -\abs{\sum_{s\in \sufijocbb{m}{n}{u}{v_0}}\lambda^{\la s, h_m\ra}}\right) +  \sum_{u \in V_m} \frac{h_m(u)}{h_n(v_0)}\abs{\sum_{s\in \sufijocbb{m}{n}{u}{v_0}}\lambda^{\la s, h_m\ra}} \\
  &\geq \varepsilon + \sum_{u \in V_m} \frac{h_m(u)}{h_n(v_0)}\abs{\sum_{s\in \sufijocbb{m}{n}{u}{v_0}}\lambda^{\la s, h_m\ra}} \\
  &\geq \varepsilon + \sum_{u \in I_\mu } \frac{h_m(u)}{h_n(v_0)}\abs{\sum_{s\in \sufijocbb{m}{n}{u}{v_0}}\lambda^{\la s, h_m\ra}} > 1 \ , \\
\end{align*}
}
where in the first equality we have used \eqref{eq:blabla1} and the fact that for all $u \in V_m$  
$$\abs{\sum_{s\in \sufijocbb{m}{n}{u}{v_0}}\lambda^{\la s, h_m\ra}} \leq \matrizp{m}{n}{u}{v_0},$$ and in the last inequality we have used \eqref{eq:blabla}.
This is a contradiction and (2) follows. 

\subsubsection{Proof of the necessity of the conditions}
We start noticing that $B_{m}(u)$ for $m\geq 1$ and $u \in V_m$ can be written as a disjoint union of elements of $\mathcal{P}_n$ for $n>m$ in the following way
\begin{equation}\label{eq:union_disjunta}
B_{m}(u) = \bigcup_{v\in V_n} \bigcup_{s\in \sufijocbb{m}{n}{u}{v}} T^{-\la s,h_m\ra}B_n(v).
\end{equation}

Applying $\mu$ on both sides of \eqref{eq:union_disjunta} and introducing heights to get measures of towers we have
\begin{equation}\label{eq:medidas_torres}
\mu(\tau_m=u)=\sum_{v \in V_n}\frac{h_m(u)}{h_n(v)}P_{m,n}(u,v)\mu(\tau_n=v).
\end{equation}

On the other side, we integrate a fixed eigenfunction $f$ of modulus $1$ associated to $\lambda$ over $B_m(u)$. 
We use equality \eqref{eq:union_disjunta} to obtain
\begin{align}
\int_{B_m(u)}fd\mu &= \sum_{v \in V_n}\sum_{s\in \sufijocbb{m}{n}{u}{v}}\int_{B_n(v)}f\circ T^{-\la s,h_m\ra}d\mu \nonumber \\
  &= \sum_{v\in V_n}\left(\sum_{s\in \sufijocbb{m}{n}{u}{v}}\!\!\!\lambda^{-\la s,h_m\ra}\right)\int_{B_n(v)}f d\mu \ ,  \label{eq:integrando_funcionpropia}
\end{align}

and then, applying \eqref{eq:definicion_c_ro} on \eqref{eq:integrando_funcionpropia} and multiplying both sides by $h_m(u)$ we get
\begin{align}
& \mu(\tau_m=u)c_m(u)\lambda^{-\rho_m(u)} \nonumber \\
&\qquad\qquad = \sum_{v\in V_n}\frac{h_m(u)}{h_n(v)}\left(\sum_{s\in \sufijocbb{m}{n}{u}{v}}\!\!\!\lambda^{-\la s,h_m\ra}\right)
\mu(\tau_n=v)c_n(v)\lambda^{-\rho_n(v)}. \label{eq:medidas_torres_mas_c_rho_lambda}
\end{align}

With these two very similar equations, \eqref{eq:medidas_torres} and \eqref{eq:medidas_torres_mas_c_rho_lambda}, we can conclude the ``{}necessity part''{} of the proof in the following way.
First, we have the inequalities
\begin{equation}
|c_n (v) |\leq 1 \hbox{ and } \bigg|\!\sum_{s\in \sufijocbb{m}{n}{u}{v}}\!\!\!\lambda^{-\la s,h_m\ra}\bigg|\leq P_{m,n}(u,v),
\end{equation}
so we take the absolute value on both sides of \eqref{eq:medidas_torres_mas_c_rho_lambda} to obtain
\begin{align*}
\mu(\tau_m=u)c_m(u) &\leq \sum_{v\in V_n}\frac{h_m(u)}{h_n(v)}\bigg|\!\sum_{s\in \sufijocbb{m}{n}{u}{v}}\!\!\!\lambda^{-\la s,h_m\ra}\bigg|\mu(\tau_n=v) \\
  &\leq \sum_{v \in V_n}\frac{h_m(u)}{h_n(v)}P_{m,n}(u,v)\mu(\tau_n=v) \\
  &= \mu(\tau_m=u).
\end{align*}
Then, applying Lemma \ref{lemma:convergencia_c} we see that
\begin{equation}
\sum_{v \in V_n}\frac{h_m(u)}{h_n(v)}\left[P_{m,n}(u,v)-\bigg|\!\sum_{s\in \sufijocbb{m}{n}{u}{v}}\!\!\!\lambda^{-\la s,h_m\ra}\bigg|\right]\mu(\tau_n=v)\xrightarrow[m\to+\infty]{}0
\end{equation}
uniformly for $n>m$.

Finally, for any $v\in I_{\mu}$ we get
$$
\frac{h_m(u)}{h_n(v)}\left[P_{m,n}(u,v)-\bigg|\!\sum_{s\in \sufijocbb{m}{n}{u}{v}}\!\!\!\lambda^{-\la s,h_m\ra}\bigg|\right]\xrightarrow[m\to+\infty]{}0
$$
uniformly for $n>m$, which is condition (2) of Theorem \ref{theo:cns_sinrho} (recall that $u$ is an arbitrary vertex in $V_m$).

\subsubsection{Proof of the sufficiency of the conditions}
Now we assume $$\frac{h_m(u)}{h_n(v)}\left[\matrizp{m}{n}{u}{v}-\bigg|\!\!\sum_{s\in\sufijocbb{m}{n}{u}{v}}\!\!\!\lambda^{\la s,h_m\ra}\bigg|\right] \xrightarrow[m\to+\infty]{}0$$ uniformly in $n>m\geq 1$ for $u \in V_m$ and $v \in I_\mu$.

We start with the following lemma, which will allow us to handle the sum of powers of $\lambda$ that appear in the (equivalent) conditions of Theorem \ref{theo:cns_sinrho}.

\begin{lemma}[Geometric Lemma]\label{lemma:geometrico}
For $N> 1$ and $k=1,\ldots,N$ consider complex numbers $z_k=\exp(2i \pi \alpha_k)$ with $\alpha_k\in [0,1)$. Let $\varepsilon\leq 1$ and $\gamma\leq 1/2$ be positive real numbers.
If $\ds \Big |\sum_{k=1}^{N}z_k\Big |>(1-\varepsilon)N$, then there exists $1\leq \ell\leq N$ such that
$$
\#\set{\alpha_k \talque 1\leq k\leq N \y \vvert \alpha_k - \alpha_{\ell} \vvert \geq \gamma} <\frac{2N\varepsilon}{1-\cos(2\pi\gamma)}.
$$
\end{lemma}
\begin{proof}

Set $S=\sum_{k=1}^{N}z_k$. It is easy to check that $S$ satisfies
$$
|S|^{2} = N + 2\!\!\!\sum_{1\leq i<j\leq N}\!\!\cos(2\pi\vvert \alpha_i - \alpha_j\vvert).
$$
Then, from $|S|>(1-\varepsilon)N $ one gets $1-|S|^{2}/N^{2} < 2\varepsilon $, and thus 
\begin{align}
2\!\!\!\sum_{1\leq i<j\leq N}\!\!\cos(2\pi\vvert \alpha_i - \alpha_j\vvert) > N(N-1) -2N^{2}\varepsilon. \label{eq:geom1}
\end{align}
On the other hand, if we define
$$
K=\set{(\alpha_i, \alpha_j) \talque 1\leq i,j\leq N \y \vvert \alpha_i - \alpha_j \vvert \geq \gamma},
$$
then
\begin{align}
& 2\!\!\!\sum_{1\leq i<j\leq N}\!\!\cos(2\pi\vvert \alpha_i - \alpha_j\vvert) \nonumber \\
=&\sum_{1\leq i\neq j\leq N}\!\!\cos(2\pi\vvert \alpha_i - \alpha_j\vvert) \nonumber \\
=& \sum_{\substack{(\alpha_i,\alpha_j)\in K \\ 1\leq i,j\leq N}}\!\!\cos(2\pi\vvert \alpha_i - \alpha_j\vvert) + \sum_{\substack{(\alpha_i,\alpha_j)\not\in K \\ 1\leq i\neq j\leq N}}\!\!\cos(2\pi\vvert \alpha_i - \alpha_j\vvert) \nonumber \\
\leq &\# K\cos(2\pi\gamma) + N(N-1) - \#K.
\end{align}

Combining this last inequality with \eqref{eq:geom1} we deduce that
$$
\# K < \frac{2N^{2}\varepsilon}{1-\cos(2\pi\gamma)}.
$$
So there should exist $1\leq \ell\leq N$ such that
$$
\#\set{\alpha_k \talque 1\leq k\leq N \y \vvert \alpha_k - \alpha_{\ell} \vvert \geq \gamma} <\frac{2N\varepsilon}{1-\cos(2\pi\gamma)}.
$$
\end{proof}

Notice that the condition of Theorem \ref{theo:cns_sinrho} is invariant by telescoping. 
Then, without loss of generality from now on we will make, by telescoping the associated Bratteli-Vershik diagram if necessary, the following assumptions: 
\begin{itemize}
\item[(1)] For all $u\not\in I_{\mu}$,
\begin{equation}\label{eq:torres_despreciables_medidas_sumables}
\sum_{n \geq 1} \mu(\tau_n=u) < +\infty
\end{equation}
(not only $\mu(\tau_n=u)\xrightarrow[n\to+\infty]{}0$ as in the definition of clean diagram).
\item[(2)] For all $n>m\geq 1$ and $u,v\in I_{\mu}$,
\begin{equation}\label{eq:cota_inf_cuocientes}
\frac{h_m(u)}{h_n(v)}\matrizp{m}{n}{u}{v} > \frac{\delta_0}{2},
\end{equation} 
where $\delta_0$ is the constant from the definition of clean diagram. It is clear that part (2) of Lemma \ref{lemma:comportamiento_cuocientes} allows us to assume this fact.
\end{itemize}

The proof of the ``{}sufficiency part''{} of Theorem \ref{theo:cns_sinrho} consists in constructing functions $\rho_n:V_n\to \RR$, with $n\geq 1$, as in the formulation of Theorem \ref{theo:cns_old_rho}, and proving the convergence associated with them. In order to accomplish this, we will break the proof below into several steps. 
\medskip

The first step consists in constructing with the help of Lemma \ref{lemma:geometrico} some useful sequences $(\Lambda_{m,n}(u,v) \talque n>m\geq 1 \y u,v\in I_{\mu})$ and $(\mathcal{D}_{m,n} \talque n>m\geq 1)$ of integers and measurable sets respectively, and show some relevant properties.

For $n>m\geq 1$ and $u,v\in I_{\mu}$ write
\begin{equation}\label{eq:def_Delta}
\Delta_{m,n}(u,v) = \frac{h_m(u)}{h_n(v)}\left[\matrizp{m}{n}{u}{v}-\bigg|\!\!\sum_{s\in\sufijocbb{m}{n}{u}{v}}\!\!\!\lambda^{\la s,h_m\ra}\bigg|\right].
\end{equation}

Using \eqref{eq:cota_inf_cuocientes} we can see that
$$
\Delta_{m,n}(u,v) > \frac{\delta_0}{2}\left[1-\frac{1}{\matrizp{m}{n}{u}{v}}\bigg|\!\!\sum_{s\in\sufijocbb{m}{n}{u}{v}}\!\!\!\lambda^{\la s,h_m\ra}\bigg|\right],
$$
and therefore
\begin{equation}\label{eq:desigualdad_lambda_Delta_P}
\bigg|\!\!\sum_{s\in\sufijocbb{m}{n}{u}{v}}\!\!\!\lambda^{\la s,h_m\ra}\bigg|>\left(1-\frac{2\Delta_{m,n}(u,v)}{\delta_0}\right)\!\matrizp{m}{n}{u}{v}.
\end{equation}

Now, consider
\begin{equation}\label{eq:def_gamma_mn}
\gamma=\gamma_{m,n}(u,v) = \frac{1}{2\pi}\arccos\left(1-\sqrt{\Delta_{m,n}(u,v)}\right),
\ \varepsilon=\varepsilon_{m,n}(u,v) =\frac{2\Delta_{m,n}(u,v)}{\delta_0}.
\end{equation}
Notice that with this choice of $\gamma$ and $\varepsilon$, if we take  large enough values of $n > m \geq 1$ and we use inequality \eqref{eq:desigualdad_lambda_Delta_P}, then the hypotheses of Lemma \ref{lemma:geometrico} hold for the complex numbers $\lambda^{\la s,h_m\ra}$ (recall that we are assuming $\Delta_{m,n}(u,v)\xrightarrow[m\to+\infty]{} 0$ uniformly for $n>m\geq 1$, for all $u,v\in I_{\mu}$). We deduce that there exists 
$s^*_{m,n}(u,v) \in\sufijocbb{m}{n}{u}{v}$ such that
\begin{align}
& \#\set{s\in\sufijocbb{m}{n}{u}{v} \talque \vvert\alpha\la s,h_m \ra - \alpha\la s^*_{m,n}(u,v),h_m \ra\vvert \geq \gamma_{m,n}(u,v)} \nonumber \\
< &  \frac{4\matrizp{m}{n}{u}{v}}{\delta_0}\sqrt{\Delta_{m,n}(u,v)}. \label{eq:lema_geom_aplicado}
\end{align}

It is important to remark that $s^*_{m,n}(u,v)$ is chosen arbitrarily for ``{}not large enough''{} values of $n>m \ge 1$.
\smallskip

For $n>m\geq 1$ and $u,v\in I_{\mu}$, write $\Lambda_{m,n}(u,v)=\la s^*_{m,n}(u,v),h_m\ra$ and define $\mathcal{D}_{m,n}$ as the set of points $x \in X$ such that 
$\tau_m (x), \tau_n(x)\in I_\mu$ and
$$
\vvert\alpha\la s_{m,n}(x),h_m\ra - \alpha\Lambda_{m,n}(\tau_m(x),\tau_n(x))\vvert \geq \gamma_{m,n}(\tau_m(x),\tau_n(x)).
$$

\begin{lemma}\label{lemma:mu_D_converge_a_cero}
$\ds\mu(\mathcal{D}_{m,n})\xrightarrow[m\to+\infty]{}0$ uniformly for $n>m\geq 1$.
\end{lemma}

\begin{proof}
For $n>m\geq 1 $ the measure of $\mathcal{D}_{m,n}$ can be written
$$
\sum_{u\in I_{\mu}}\sum_{v\in I_{\mu}}\mu\{x\in X; \vvert\alpha\la s_{m,n}(x), h_m \ra - \alpha\Lambda_{m,n}(u,v)\vvert
\geq \gamma_{m,n}(u,v), \tau_m(x) = u, \tau_n(x)=v\}.
$$
Assuming $m$ is large enough, from inequality \eqref{eq:lema_geom_aplicado} we obtain
\begin{align}
\mu(\mathcal{D}_{m,n})  &< \sum_{u\in I_{\mu}}\sum_{v\in I_{\mu}}\frac{4\matrizp{m}{n}{u}{v}}{\delta_0}\sqrt{\Delta_{m,n}(u,v)}\mu_{n}(v)h_m(u) \nonumber\\
 &= \sum_{u\in I_{\mu}}\sum_{v\in I_{\mu}}\frac{4h_m(u)\matrizp{m}{n}{u}{v}}{\delta_0 h_n(v)}\sqrt{\Delta_{m,n}(u,v)}\mu(\tau_n=v) \nonumber \\
 &\label{eq:cota_Dmn} \leq \sum_{u\in I_{\mu}}\sum_{v\in I_{\mu}}\frac{4\sqrt{\Delta_{m,n}(u,v)}}{\delta_0}\mu(\tau_n=v).
\end{align}
The lemma follows since, by hypothesis, the right hand side goes to zero as desired.
\end{proof}

\begin{lemma}[Quasi-Additivity of $\alpha\Lambda_{m,n}(u,v)$]\label{lemma:cuasi_aditividad_Lambda}
For $n>m>\ell \geq 1$ large enough and $u,v,w \in I_{\mu}$ we have
\begin{align*}
\vvert\alpha\Lambda_{\ell,m}(u,v) + \alpha\Lambda_{m,n}(v,w) - \alpha\Lambda_{\ell,n}(u,w)\vvert   < \gamma_{\ell,m}(u,v) + \gamma_{m,n}(v,w) + \gamma_{\ell,n}(u,w).
\end{align*}
\end{lemma}

\begin{proof}
Fix $u,v,w \in I_{\mu}$. For $1 \leq \ell < m <n$ write
\begin{align*}
\mathcal{V}_{\ell,m,n} = & \set{x\in X; \tau_{\ell}(x)=u, \tau_{m}(x)=v, \tau_{n}(x)=w} \\
&\cap \set{x\in X;\vvert\alpha\la s_{\ell, m}(x), h_{\ell}\ra - \alpha\Lambda_{\ell,m}(u,v)\vvert < \gamma_{\ell,m}(u,v)} \\
&\cap \set{x\in X; \vvert\alpha\la s_{m, n}(x), h_{m}\ra - \alpha\Lambda_{m, n}(v,w)\vvert < \gamma_{m,n}(v,w)}. \\
\end{align*}
First, we estimate the measure of this set. Since the maps $s_{\ell, m}(\cdot)$ and $s_{m, n}(\cdot)$ only depend on levels from $\ell$ to $m$ and $m$ to $n$ respectively,  we can see from the structure of the measure $\mu$ that $\mu(\mathcal{V}_{\ell,m,n})$ is equal to
\begin{align*}
 &\#\set{s\in\sufijocbb{\ell}{m}{u}{v} \talque \vvert\alpha\la s, h_{\ell}\ra - \alpha\Lambda_{\ell,m}(u,v)\vvert < \gamma_{\ell,m}(u,v)} \\
 &\qquad\times \#\set{s\in\sufijocbb{m}{n}{v}{w} \talque \vvert\alpha\la s, h_{m}\ra - \alpha\Lambda_{m,n}(v,w)\vvert < \gamma_{m, n}(v,w)} \\
 &\qquad\times\mu_n(w)h_{\ell}(u).
\end{align*}
For $1\leq \ell<m<n$ large enough, inequality \eqref{eq:lema_geom_aplicado} shows that the two set cardinalities involved in the above expression can be bounded below by $\matrizp{\ell}{m}{u}{v}/2$ and $\matrizp{m}{n}{v}{w}/2$ respectively. So, there exists $n_0$ such that, for $n>m>\ell>n_0$,
\begin{align}
\mu(\mathcal{V}_{\ell,m,n}) &\geq \frac{\matrizp{\ell}{m}{u}{v}\matrizp{m}{n}{v}{w}\mu_n(w)h_{\ell}(u)}{4} \nonumber \\
 &= \frac{\mu(\tau_n=w)}{4}\frac{h_{\ell}(u)}{h_{m}(v)}\matrizp{\ell}{m}{u}{v}\frac{h_{m}(v)}{h_{n}(w)}\matrizp{m}{n}{v}{w} \nonumber \\
 &\geq \frac{\delta_0^{3}}{16}, \nonumber
\end{align}
where in the last inequality we have used \eqref{eq:cota_inf_cuocientes} and the fact that the diagram is clean. Also, $n_0$ can be chosen such that  $\mu(\mathcal{D}_{\ell,n}) < \delta_0^{3}/16$, as a consequence of Lemma \ref{lemma:mu_D_converge_a_cero}.

Now we proceed by contradiction. Suppose that the assertion of the lemma is false for the fixed $u,v,w\in I_{\mu}$. Then we can find positive integers $n>m>\ell>n_0$ such that
\begin{align*}
\vvert\alpha\Lambda_{\ell,m}(u,v) + \alpha\Lambda_{m,n}(v,w) - \alpha\Lambda_{\ell,n}(u,w)\vvert 
\geq \gamma_{\ell,m}(u,v) + \gamma_{m,n}(v,w) + \gamma_{\ell,n}(u,w).
\end{align*}
We claim that for these positive integers we have $\mathcal{V}_{\ell,m,n}\subseteq \mathcal{D}_{\ell,n}$. Indeed, for any $x_0\in \mathcal{V}_{\ell,m,n}$, using \eqref{eq:formulareturn_mn_suma} we get
\begin{align*}
& \vvert \alpha \la s_{\ell,n}(x_0), h_{\ell} \ra - \alpha\Lambda_{\ell,n}(\tau_{\ell}(x_0),\tau_n(x_0))\vvert  \\
= & \vvert \alpha \la s_{\ell,n}(x_0), h_{\ell} \ra - \alpha\Lambda_{\ell,n}(u,w)\vvert \\
 \geq & \vvert\alpha\Lambda_{\ell,m}(u,v) + \alpha\Lambda_{m,n}(v,w) - \alpha\Lambda_{\ell,n}(u,w)\vvert \\
  & - \vvert\alpha\Lambda_{\ell,m}(u,v)-\alpha\la s_{\ell,m}(x_0), h_{\ell}\ra\vvert - \vvert\alpha\Lambda_{m,n}(v,w)-\alpha\la s_{m,n}(x_0), h_{m}\ra\vvert \\
  > & \gamma_{\ell,n}(u,w),
\end{align*}
{\it i.e.}, $x_0\in \mathcal{D}_{\ell,n}$.
The inclusion $\mathcal{V}_{\ell,m,n}\subseteq \mathcal{D}_{\ell,n}$ contradicts the fact that $\mu(\mathcal{D}_{\ell,n}) < \delta_0^{3}/16 \leq \mu(\mathcal{V}_{\ell,m,n})$.

This proves the lemma noticing that we have only a finite number of different choices for $u,v,w\in I_{\mu}$.
\end{proof}

Our next task consists in defining a suitable set of full measure (the complement of a set $\mathcal{C}$ of null measure) such that we can handle the size of $\vvert \alpha\la s_{m,n}(\cdot), h_m\ra - \alpha\Lambda_{m,n}(\tau_m(\cdot),\tau_n(\cdot))\vvert$ for any of their elements.
To do this, fix a decreasing sequence of positive real numbers $(\varepsilon_n \talque n\geq 1)$ such that $\sum_{n\geq 1} \varepsilon_n <\infty$. By \eqref{eq:def_gamma_mn} and the hypothesis, we get that for all $u,v\in I_{\mu}$, $\gamma_{m,n}(u,v) \xrightarrow[m\to+\infty]{} 0$ uniformly for $n>m$. 
Hence, we can fix an increasing sequence of positive integers $(m_k \talque k\geq 1)$ such that for every $n> m_k$ and $u,v\in I_{\mu}$
\begin{equation}\label{gamma_sumable}
\gamma_{m_k,n}(u,v) \leq \varepsilon_k \ .
\end{equation}

For $n>m\geq 1$ set
$$
\mathcal{C}_{m,n} = \mathcal{D}_{m,n} \cup \set{x\in X;\tau_m(x)\not\in I_{\mu}} \cup \set{x\in X;\tau_n(x)\not\in I_{\mu}}.
$$

\begin{lemma}\label{lemma:medida_C_es_nula}
Let $(m_k \talque k\geq 1)$ be the above-mentioned sequence \ and set $\ds\mathcal{C}=\limsup_{k\to+\infty}\mathcal{C}_{m_k,m_{k+1}}$. Then, $\mu(\mathcal{C})=0$.
\end{lemma}
\begin{proof}
Notice that for $n>m$ large enough and every $u, v \in I_{\mu}$ we have
$$
\Delta_{m,n}(u,v) < \sqrt{\Delta_{m,n}(u,v)} < \gamma_{m,n}(u,v).
$$
Hence $\sqrt{\Delta_{m_k,m_{k+1}}(u,v)} < \varepsilon_k$ for every $k\geq 1$ and $u,v \in I_{\mu}$. Using the bound \eqref{eq:cota_Dmn}, one can obtain the summability of $(\mu(\mathcal{D}_{m_k,m_{k+1}}) \talque k\geq 1)$. Therefore

$$
\sum_{k\geq 1} \mu(\mathcal{C}_{m_k,m_{k+1}}) < \infty
$$
(recall \eqref{eq:torres_despreciables_medidas_sumables}), and the lemma follows by Borel-Cantelli.
\end{proof}

Finally, we proceed to construct the sequence $(\rho_m\talque m\geq 1)$ of Theorem \ref{theo:cns_old_rho} which is part of the main goal of this proof. To this end, fix $v_0\in I_{\mu}$, and by means of a standard diagonalization process we can find $(n_i\talque i\geq 1)$ such that for all $m\geq 1$ and $u\in I_{\mu}$, the sequence $(\alpha\Lambda_{m,n_i}(u,v_0)\talque i\geq 1)$ is convergent $\operatorname{mod}\; \ZZ$.
Considering this, we define
$$
\rho_m(u) = \left\{
\begin{array}{ccc}
\ds \frac{1}{\alpha}\lim_{i\to+\infty}\alpha\Lambda_{m,n_i}(u,v_0)\; (\operatorname{mod}\; \ZZ) &{\rm for}& u\in I_{\mu} \\
0&{\rm for}& u\not\in I_{\mu}.
\end{array}
\right.
$$

By Theorem \ref{theo:cns_old_rho}, we will establish the ``{}sufficient part''{} of Theorem \ref{theo:cns_sinrho} if we prove the following lemma.

\begin{lemma}\label{lemma:convergencia_rho}
Let $(m_k \talque k\geq 1)$ and $\mathcal{C}$ be as in the formulation of Lemma \ref{lemma:medida_C_es_nula}. Then, for all $x\in X\setminus \mathcal{C}$, 
$$
\left(\alpha(r_{m_k}(x) + \rho_{m_k}(\tau_{m_k}(x))) \talque k\geq 1\right)
$$
converges mod $\ZZ$.
\end{lemma}
\begin{proof}
Take $\varepsilon>0$ and $x\not\in\mathcal{C}$. There exists a positive integer $k_0$ such that $x\not\in\mathcal{C}_{m_k,m_{k+1}}$ for all $k\geq k_0$. Here and subsequently $u_k$ denotes the vertex $\tau_{m_k}(x)$. By definition of $\mathcal{C}_{m_k,m_{k+1}}$ we have $u_k\in I_{\mu}$ and for $k\geq k_0$
\begin{equation}\label{eq:desigualdad_miembros_conjunto_bueno}
\vvert \alpha\la s_{m_{k},m_{k+1}}(x), 	h_{m_k}\ra  - \alpha\Lambda_{m_k,m_{k+1}}(u_k,u_{k+1})\vvert < \gamma_{m_k,m_{k+1}}(u_k,u_{k+1}).
\end{equation}

The integer $k_0$ will be chosen large enough such that $\sum_{k=k_0}^{\infty}\varepsilon_k < \varepsilon/8$ (recall the sequence $(\varepsilon_k\talque k\geq 1)$ from the construction of sequence $(m_k\talque k\geq 1)$ is summable).

Now fix $\ell > k \geq k_0$. By definition of the $\rho_m$'s we can find an integer $j\geq 0$ such that simultaneously
\begin{align}
\vvert \alpha\Lambda_{m_k,j}(u_k,v_0) - \alpha\rho_{m_k}(u_k)\vvert &< \varepsilon/4 \; \; \; {\rm and}\nonumber \\
\vvert \alpha\Lambda_{m_{\ell},j}(u_{\ell},v_0) - \alpha\rho_{m_{\ell}}(u_{\ell})\vvert &< \label{eq:Lambda_cerca_rho}\varepsilon/4.
\end{align}

For $k\leq i < \ell$ define
\begin{align*}
\Theta_i &= \alpha\Lambda_{m_i,m_{i+1}}(u_i,u_{i+1}) + \alpha\Lambda_{m_{i+1},j}(u_{i+1},v_0) - \alpha\Lambda_{m_i,j}(u_i,v_0), \\
\Omega_i &= \alpha\la s_{m_i,m_{i+1}}(x), h_{m_i}\ra - \alpha\Lambda_{m_i,m_{i+1}}(u_i,u_{i+1}).
\end{align*}

From the quasi additivity stated in Lemma \ref{lemma:cuasi_aditividad_Lambda} and \eqref{eq:desigualdad_miembros_conjunto_bueno} we have, for $k\leq i < \ell$,
\begin{align}
\vvert\Theta_i\vvert & < \gamma_{m_i,m_{i+1}}(u_i,u_{i+1}) + \gamma_{m_{i+1},j}(u_{i+1},v_0) + \gamma_{m_i,j}(u_i,v_0) \nonumber \\
 & < \varepsilon_i + \varepsilon_{i+1} + \varepsilon_i < 3\varepsilon_i, \label{eq:desigualdad_Theta_i}\\
\vvert\Omega_i\vvert & < \gamma_{m_i,m_{i+1}}(u_i,u_{i+1}) < \varepsilon_i \label{eq:desigualdad_Omega_i}
\end{align}
(notice that we could have chosen $k_0$ large enough in order to apply Lemma \ref{lemma:cuasi_aditividad_Lambda}), and with the help of properties \eqref{eq:formulareturn_mn} and \eqref{eq:suma_sufijos} of suffix vectors, we deduce the identity
\begin{equation}\label{eq:identidad_Theta_Omega_i}
\sum_{i=k}^{\ell-1}\Theta_i + \Omega_i  = \alpha \left(r_{m_{\ell}}(x) - r_{m_{k}}(x)\right) + \alpha\Lambda_{m_{\ell},j}(u_{\ell},v_0) - \alpha\Lambda_{m_{k},j}(u_{k},v_0).
\end{equation}

Combining \eqref{eq:Lambda_cerca_rho}, \eqref{eq:desigualdad_Theta_i}, \eqref{eq:desigualdad_Omega_i} and \eqref{eq:identidad_Theta_Omega_i} gives
\begin{align*}
\vvert \alpha\left(r_{m_{\ell}}(x) + \rho_{m_{\ell}}(\tau_{m_{\ell}}(x))\right) - \alpha\left(r_{m_{k}}(x) + \rho_{m_{k}}(\tau_{m_{k}}(x))\right)\vvert &  \\
 & \hspace{-9em}< \frac{\varepsilon}{2} + \sum_{i=k}^{\ell-1} \vvert\Theta_i\vvert + \vvert\Omega_i\vvert \\
 & \hspace{-9em}< \frac{\varepsilon}{2} + \sum_{i=k}^{\ell-1} 4\varepsilon_i < \frac{\varepsilon}{2} + 4\sum_{i=k}^{\infty} \varepsilon_i < \varepsilon.
\end{align*}
Therefore, $(\alpha\left(r_{m_{k}}(x) + \rho_{m_{k}}(\tau_{m_{k}}(x))\right)\talque k\geq 1)$
is a Cauchy sequence. This proves the lemma and consequently, Theorem \ref{theo:cns_sinrho}.
\end{proof}

The criterion of Theorem \ref{theo:cns_sinrho} can be simplified as folows.

\begin{coro}\label{coro:cs_sin_rho}
Let $(X,T)$ be a minimal Cantor  system given by a Bratteli-Vershik representation as in Theorem \ref{theo:cns_sinrho} and let $\mu$ be one of its ergodic probability measures. 
Let us take a complex number $\lambda$ of modulus 1. If for all $u,v\in I_{\mu}$
\begin{align}\label{eq:cs_sin_rho}
\frac{\displaystyle \left|\sum_{s\in\sufijocbb{m}{n}{u}{v}}\lambda^{\la s,h_m\ra}\right|}{\matrizp{m}{n}{u}{v}}\xrightarrow[m\to+\infty]{} 1
\end{align}
uniformly for $n>m$, then $\lambda$ is an eigenvalue of $(X,T)$ for $\mu$. The converse is also true when the Bratteli-Vershik representation of $(X,T)$ satisfies condition \eqref{eq:cota_inf_cuocientes}. 
\end{coro}
\begin{proof}
We will show that condition (2) of Theorem \ref{theo:cns_sinrho} holds.

For $u\in\{1,\dots, d\}$ and $v\in I_{\mu}$
\begin{align*}
&\frac{h_m(u)}{h_n(v)}\left[\matrizp{m}{n}{u}{v} - \left|\sum_{s\in\sufijocbb{m}{n}{u}{v}}\lambda^{\la s,h_m\ra}\right|\right] \\ &\qquad\qquad = \frac{h_m(u)}{h_n(v)}\matrizp{m}{n}{u}{v}\left[1-\frac{\left|\sum_{s\in\sufijocbb{m}{n}{u}{v}}\lambda^{\la s,h_m\ra}\right|}{\matrizp{m}{n}{u}{v}}\right].
\end{align*}
If $u\not\in I_{\mu}$ the first factor of the right hand side converges to 0 in $m$, uniformly for $m>n$, because of part (1) of Lemma \ref{lemma:comportamiento_cuocientes}, while the second factor remains bounded. 
If $u\in I_{\mu}$ the convergence to 0 of the left hand is implied from the hypotheses of this Corollary. 

It is immediate that the converse is true when the Bratteli-Vershik representation satisfies condition \eqref{eq:cota_inf_cuocientes}
\end{proof}

\subsection{A necessary and sufficient condition in the exact finite rank case}
Systems of exact finite rank were introduced in \cite{bkms}, they are uniquely ergodic and are defined asserting that $I_\mu = \{ 1, \ldots , d\}$ for the unique ergodic measure $\mu$.

With respect to eigenvalues, the general necessary and sufficient condition for linearly recurrent minimal Cantor systems proposed in \cite{necesariasuficiente} is described in the form of a convergence of a series. These systems are uniquely ergodic and even further, from \cite[Lemma~4]{lr} it follows that they are of exact finite rank. 

In the case of finite rank minimal Cantor systems a similar condition was shown to be necessary in Proposition 18 of \cite{rangofinito}, making use of the existence of auxiliary functions $\rho_n:V_n \to \RR$ as those in Theorem \ref{theo:cns_old_rho}. 
One virtue of this condition is that it does not need to handle a uniform convergences in two indices as in Theorem \ref{theo:cns_sinrho}. 
Nevertheless, the auxiliary functions could be difficult to compute as was observed in the proof of previous section. 

Here, for a convenient representation of the system, we prove that the necessary condition to be an eigenvalue given in \cite{rangofinito} is actually sufficient in the exact finite rank case.

We will say that a clean Bratteli-Vershik representation is \emph{stable} if condition \eqref{eq:cota_inf_cuocientes}
 holds, {\emph {i.e.}}, for all $1 \leq m < n$ and $u,v\in I_{\mu}$, $h_m(u)\matrizp{m}{n}{u}{v}/h_n(v) > \delta_0/2$, where $\delta_0$ is the constant coming from the definition of clean property. 
We observe that it is always possible to get a stable representation of a minimal Cantor system thanks to Lemma \ref{lemma:comportamiento_cuocientes}.
This condition, as was illustrated in the proof of previous theorem, recovers in a better way the behaviour of invariant measures for the matrices $M_n$ of the Bratteli-Vershik representation. 

\begin{theo} 
\label{theo:cnsvpmedible}
Let $(X,T)$ be a minimal Cantor system given by a proper and stable Bratteli-Vershik representation $B=(V,E,\preceq)$ of exact finite rank $d$ for  the ergodic probability measure $\mu$. 
Then, $\lambda=\exp(2i \pi \alpha)$ is an eigenvalue of $(X,T)$ for $\mu$ if and only if for every $n\geq 1$ there exist functions $\rho_{n}: V_n \to \RR$ such that
\begin{equation}
\label{eq:cnsvpmedible}
\sum_{n\geq 1}  \frac{1}{M_{n+1}(u_{n},u_{n+1})} \sum_{s\in S_{n}(u_{n}, u_{n+1})} 
\left |\lambda^{\langle s,h_{n} \rangle-\rho_{n+1}(u_{n+1})+\rho_{n}(u_{n})}-1 \right |^{2} < +\infty
\end{equation}
for any sequence $(u_{n}; n\geq 1)$ in $I_{\mu}$.
\end{theo}
\vskip 0.5cm

To prove condition \eqref{eq:cnsvpmedible} is sufficient for $\lambda$ to be an eigenvalue of $(X,T)$ for $\mu$ we follow the same strategy developed in \cite{necesariasuficiente}. The proof will be a consequence of several steps developed in the next subsections where we follow notations in Theorem \ref{theo:cnsvpmedible} and we assume \eqref{eq:cnsvpmedible} holds.

\subsubsection{An alternative formulation} For $n\geq 1$ let 
us define $\theta_{n}(s,u_{n},u_{n+1})$ to be the fractional part in $(-1/2,1/2]$ of $\alpha (\langle s,h_{n} \rangle-\rho_{n+1}(u_{n+1})+\rho_{n}(u_{n}))$. Clearly, 
$\vvert \theta_{n}(s,u_{n},u_{n+1}) \vvert=| \theta_{n}(s,u_{n},u_{n+1})|$.
We have
\begin{align*}
|e^{2i\pi\theta_{n}(s,u_{n},u_{n+1})}-1|^{2} 
&=2 (1-\cos( 2\pi\theta_{n}(s,u_{n},u_{n+1}))). 
\end{align*}

The function defined by $f(t)=t^{2}/\left(2\left(1-\cos(2\pi t)\right)\right)$ for $t\in [-1/2,1/2]\setminus \left\{0 \right\}$ is even and strictly positive. Also, we can define $f(0)=1/(4\pi^2)>0$ and then $f$ is continuous and strictly positive on $[-1/2,1/2]$. By the compactness of $[-1/2,1/2]$ there exist strictly positive constants $C_1$ and $C_2$ such that $f\left([-1/2,1/2]\right) \subseteq [C_1, C_2]$ (in fact $f(0)=1/(4\pi^2)\leq f(t)\leq 1/16=f(1/2)$ for $t\in [-1/2,1/2]$). So, for every $s \in \sufijocbb{n}{n+1}{u_{n}}{u_{n+1}}$ we have

$$
0 < C_1 \leq \frac{\vvert  \theta_{n}(s,u_{n},u_{n+1}) \vvert^{2}}{2 (1-\cos( 2\pi\theta_{n}(s,u_{n},u_{n+1})))} \leq C_2. 
$$
Then condition \eqref{eq:cnsvpmedible} is equivalent with 
\begin{equation}
\label{eq:cnsvpmediblebis}
\sum_{n\geq 1} \frac{1}{M_{n+1}(u_{n},u_{n+1})} \sum_{s\in \sufijocbb{n}{n+1}{u_{n}}{u_{n+1}}} 
\vvert  \theta_{n}(s,u_{n},u_{n+1}) \vvert ^{2} < +\infty \ .
\end{equation}
\medskip

In the linearly recurrent case, one gets that condition \eqref{eq:cnsvpmediblebis} is equivalent to $$\sum_{n\geq 1} \vvert \alpha h_{n} \vvert^{2} < +\infty \ ,$$
which is the necessary and sufficient condition for $\lambda$ to be an eigenvalue for any ergodic probability measure $\mu$ proved in \cite{necesariasuficiente}. 

\subsubsection{Markov process}
In \cite{necesariasuficiente} was observed that $(\tau_n; n \geq 1)$ is a Markov chain with respect to any invariant measure of $(X,T)$ (see also \cite{alistecoronel}).
Its importance is mainly due to the mixing condition given below (Lemma \ref{lem:geometricdeacrieasing}).
First we set some notations and assumptions. 

\begin{enumerate}
\item For integers $1\leq m<n$ define the following stochastic matrices: for $u,v \in I_\mu$ 
\begin{align*}
q_{m,n}(u,v)&=\mu(\tau_{n}=v|\tau_{m}=u) \\
&=\frac{\mu_{n}(v)}{\mu_{m}(u)} \matrizp{m}{n}{u}{v} \\
&= \frac{\mu(\tau_{n}=v)}{\mu(\tau_{m}=u)}  \frac{h_{m}(u)}{h_{n}(v)}  \matrizp{m}{n}{u}{v}. 
\end{align*}
\item 
Since the representation is stable we have that for all $u,v\in I_{\mu}$ and integers $1\leq m<n$
$$q_{m,n}(u,v)\geq \delta_0^{2}/2, \quad \mu(\tau_{n}=v)\geq \delta_0 .$$

\item For $n >1$ define $\zeta(q_{n-1,n})=1-\min_{u,v\in I_{\mu}}(q_{n-1,n}(u,v))$. By (2) we have 
$\zeta(q_{n-1,n})\leq 1- \delta_0^{2}/2 < 1$. Let $\zeta=\sup_{n>1} \zeta(q_{n-1,n})$. 
\end{enumerate}

Using Lemma 5.3 in \cite{alistecoronel} we get,
\begin{lemma} For $m,n \in \NN$ with $m<n$
$$\max_{u,u',v \in I_{\mu}} \left | q_{m,n}(u,v)-q_{m,n}(u',v)\right | \leq \zeta^{n-m} \ .$$
\end{lemma}

This lemma allows to deduce the following property. 

\begin{lemma} 
\label{lem:geometricdeacrieasing}
For $m,n \in \NN$ with $m<n$ and $u,v \in I_{\mu}$ we have 
$$\left | \mu(\tau_{n}=v | \tau_{m}=u)- \mu(\tau_{n}=v) \right | \leq  \zeta^{n-m}.$$
\end{lemma}
\begin{proof}
Recall we are assuming that the system is of exact finite rank $d$. 
Let $u,v \in I_{\mu}$. We have

\begin{align*}
& |\mu(\tau_{n}=v | \tau_{m}=u)- \mu(\tau_{n}=v)| \\
=&|\mu(\tau_{n}=v | \tau_{m}=u)- \sum_{u'\in I_{\mu}} \mu(\tau_{n}=v | \tau_{m}=u') \mu(\tau_{m}=u')| \\
=&|\sum_{u'\in I_{\mu}} \left ( \mu(\tau_{n}=v | \tau_{m}=u)-(\tau_{n}=v | \tau_{m}=u') \right ) \mu(\tau_{m}=u')|\\
\leq & \sum_{u'\in I_{\mu}} \zeta^{n-m} \mu(\tau_{m}=u')= \zeta^{n-m},
\end{align*}
where we have used that the system is of exact rank and thus
$\sum_{u'\in I_{\mu}} \mu(\tau_{m}=u')=1$.
\end{proof}

\subsubsection{Fundamental random variable $X_{n}$}
Recall we are assuming $I_{\mu}=\{1,\ldots,d\}$. To make levels explicit, depending on the context we will write $V_n$ instead of $I_\mu$ or $\{1,\ldots,d\}$.  

Define for each $n\geq 1$ and $x\in X$: 
${\btheta}_{n}(x)=\theta_{n}(s_{n}(x),\tau_{n}(x),\tau_{n+1}(x))$. Consider the random variable
$X_{n}=\btheta_{n}-\EE_{\mu}(\btheta_{n})$ and decompose it as $X_{n}=Z_{n}+Y_{n}$, where 
$Z_{n}=\btheta_{n}-\EE_{\mu}(\btheta_{n}| {\mathcal T}_{n})$ and 
$Y_{n}=\EE_{\mu}(\btheta_{n}| {\mathcal T}_{n}) - \EE_{\mu}(\btheta_{n})=\EE_{\mu}(X_{n}|{\mathcal T}_{n})$. Recall that $\T_n$ is the sigma algebra generated by the Kakutani-Rohlin partition $\P_{n}$ of level $n$. 
\medskip

We prove the convergence in $L^2 (X,\mu)$ of $\sum_{n\geq 1} Z_{n}$ and $\sum_{n\geq 1} Y_{n}$, and thus of $\sum_{n\geq 1} X_{n}$. 
\medskip

\paragraph{\it $\bullet$ Some preliminary bounds:} notice that 
$$\btheta_{n}=\sum_{u\in V_{n}} \sum_{v\in V_{n+1}} \sum_{s\in \sufijocbb{n}{n+1}{u}{v}} 
\theta_{n}(s,u,v) \, \, {\bf 1}_{\{\tau_{n}=u,\tau_{n+1}=v,s_{n}=s\}}.$$ 
Then, 
\begin{align*}
&\EE_{\mu}(\btheta_{n}| {\mathcal T}_{n})\\
=&\sum_{u\in V_{n}} {\bf 1}_{\{\tau_{n}=u\}} 
\sum_{v\in V_{n+1}} \frac{h_{n}(u)\mu_{n+1}(v)}{\mu(\tau_{n}=u)}
\sum_{s\in \sufijocbb{n}{n+1}{u}{v}} \theta_{n}(s,u,v) \\
=&\sum_{u\in V_{n}} {\bf 1}_{\{\tau_{n}=u\}} 
\sum_{v\in V_{n+1}} \frac{\mu(\tau_{n+1}=v)}{\mu(\tau_{n}=u)} \frac{h_{n}(u)}{h_{n+1}(v)}\sum_{s\in \sufijocbb{n}{n+1}{u}{v}} \theta_{n}(s,u,v) \\
=&\sum_{u\in V_{n}} {\bf 1}_{\{\tau_{n}=u\}} 
\sum_{v\in V_{n+1}} \frac{\mu(\tau_{n+1}=v)}{\mu(\tau_{n}=u)} A_{n+1}(u,v) 
\frac{1}{M_{n+1}(u,v)}\sum_{s\in \sufijocbb{n}{n+1}{u}{v}} \theta_{n}(s,u,v), 
\end{align*}
where $A_{n+1}(u,v)=\frac{h_{n}(u)}{h_{n+1}(v)}\cdot M_{n+1}(u,v)$.
Thus, 
\begin{align*}
& \EE_{\mu}(\btheta_{n})= \EE_{\mu}(\EE_{\mu}(\btheta_{n}| {\mathcal T}_{n})) \\
=&\sum_{u\in V_{n}}  \mu(\tau_{n}=u) 
 \sum_{v\in V_{n+1}}  \frac{\mu(\tau_{n+1}=v)}{\mu(\tau_{n}=u)} A_{n+1}(u,v)  
  \frac{1}{M_{n+1}(u,v)}\sum_{s\in \sufijocbb{n}{n+1}{u}{v}} \theta_{n}(s,u,v)  
  .
\end{align*}
Similarly, 
\begin{align*}
& \EE_{\mu}(\btheta_{n}^{2}| {\mathcal T}_{n}) \\
= & \sum_{u\in V_{n}} {\bf 1}_{\{\tau_{n}=u\}} 
\sum_{v \in V_{n+1}} \frac{\mu(\tau_{n+1}=v)}{\mu(\tau_{n}=u)} A_{n+1}(u,v) 
\frac{1}{M_{n+1}(u,v)}\sum_{s\in \sufijocbb{n}{n+1}{u}{v}} \theta_{n}^{2}(s,u,v), 
\end{align*}
and 
\begin{align*}
& \EE_{\mu}(\btheta_{n}^{2})=\EE_{\mu}(\EE_{\mu}(\btheta_{n}^{2}| {\mathcal T}_{n}))\\
=&\sum_{u\in V_{n}}  \mu(\tau_{n}=u) 
\sum_{v\in V_{n+1}} \frac{\mu(\tau_{n+1}=v)}{\mu(\tau_{n}=u)} A_{n+1}(u,v)  \frac{1}{M_{n+1}(u,v)}\sum_{s\in \sufijocbb{n}{n+1}{u}{v}} \theta_{n}^{2}(s,u,v) .
\end{align*}

Therefore, 
\begin{align*}
|Y_{n}|=& |\EE_{\mu}(X_{n}| {\mathcal T}_{n})| \\
\leq & \sum_{u\in V_{n}} |{\bf 1}_{\{\tau_{n}=u\}}-\mu(\tau_{n}=u)|  \sum_{v\in V_{n+1}} \frac{\mu(\tau_{n+1}=v)}{\mu(\tau_{n}=u)}  A_{n+1}(u,v) \\
& \times \left |   \frac{1}{M_{n+1}(u,v)}\sum_{s\in \sufijocbb{n}{n+1}{u}{v}} \theta_{n}(s,u,v)  \right | \\
\leq &  \sum_{u\in V_{n}} 2 \sum_{v\in V_{n+1}}  \frac{1}{\delta_0} \cdot 1 \cdot \frac{1}{M_{n+1}(u,v)}\sum_{s\in \sufijocbb{n}{n+1}{u}{v}} |\theta_{n}(s,u,v) | \ , \\
\end{align*}
where $\delta_0$ is the constant defining the cleanness property. 
But, from hypothesis \eqref{eq:cnsvpmediblebis}, 
the term in the second sum is bounded by some $\xi_{n}$, where $\sum_{n \geq 1} \xi_{n}^{2} <+\infty$. We conclude that $|Y_{n}| \leq K\xi_{n},$
where $K=2d^{2}/\delta_0$. 

\medskip

Below we will manipulate a lot of constants depending neither on variables nor on indices. We will call them universal constants and denote them also by $K$. 
\medskip

For $Z_n$ a classical computation gives, 
\begin{align*}
\EE_{\mu}(Z^{2}_{n})&= \EE_{\mu}((\btheta_{n}-\EE_{\mu}(\btheta_{n}|{\mathcal T}_{n}))^{2}) =\EE_{\mu}(\btheta_{n}^{2})- \EE_{\mu}(\btheta_{n} \EE_{\mu}(\btheta_{n}|{\mathcal T}_{n})).
\end{align*}
Thus, using a similar argument as in the bound for $|Y_{n}|$, we get
\begin{align*}
\EE_{\mu}(Z^{2}_{n})&\leq \EE_{\mu}(\btheta_{n}^{2}) + |\EE_{\mu}(\btheta_{n} \EE_{\mu}(\btheta_{n}|{\mathcal T}_{n}))| \leq \EE_{\mu}(\btheta_{n}^{2}) + \EE_{\mu}(|\btheta_{n}| |\EE_{\mu}(\btheta_{n}|{\mathcal T}_{n})|) \\
& \leq K_{1} \xi_{n}^{2} + \EE_{\mu}(|\btheta_{n}|) K_{2} \xi_{n} 
\leq K_{1} \xi_{n}^{2} + K_{3}\xi_{n} K_{2} \xi_{n}= K_4 \xi_{n}^2, \\
\end{align*}
for some universal constants $K_{1}, K_{2}, K_{3}, K_4$. Therefore,
$$
\sum_{n\geq 1} \EE_{\mu}(Z^{2}_{n}) \leq K \sum_{n\geq 1} \xi_{n}^{2} < +\infty.
$$ 

\paragraph{\it $\bullet$ The series $\sum_{n\geq 1} Z_n$ converges in $L^2 (X,\mu)$}
Let $1 \leq m<n$. We have
$\EE_{\mu}(Z_{n}|{\mathcal T}_{n})=\EE_{\mu}(\btheta_{n}-\EE_{\mu}(\btheta_{n}|{\mathcal T}_{n})|{\mathcal T}_{n})=0$ and, by definition of $Z_m$, $\EE_{\mu}(Z_m|{\mathcal T}_{m+1})=Z_m$. Then, since ${\mathcal T}_{n}$ is finer than ${\mathcal T}_{m}$, we have that 
$$\EE_{\mu}(Z_{m}Z_{n})=\EE_{\mu}(\EE_{\mu}(Z_{m}Z_{n}|{\mathcal T}_{n}))=\EE_{\mu}(Z_{m}\EE_{\mu}(Z_{n}|{\mathcal T}_{n}))=0.$$
 
This implies that 
$$\langle \sum_{i=m}^{n} Z_{i}, \sum_{j=m}^{n} Z_{j} \rangle = \sum_{i=m}^{n} \sum_{j=m}^{n} \langle Z_{i}, Z_{j} \rangle = \sum_{\ell=m}^{n} \Vert Z_{\ell} \Vert^{2}_{L^2 (X,\mu)} \leq K \sum_{\ell=m}^{n} \xi^{2}_{\ell},$$
and proves that $\sum_{n\geq 1} Z_{n}$ converges in $L^2 (X,\mu)$.
\medskip 

\paragraph{\it $\bullet$ The series $\sum_{n\geq 1} Y_n$ converges in $L^2 (X,\mu)$} 
We follow the scheme given in \cite{necesariasuficiente}. 
We have $Y_{n}=\EE_{\mu}(X_{n}|{\mathcal T}_{n})=\sum_{v \in V_{n}} {\bf 1}_{\{\tau_{n}=v\}} \, y_{n}(v)$, where $y_{n}(v)$ is a constant value. 
Let $0\leq k \leq n$. Then, 
\begin{align*}
\EE_{\mu}(Y_{n}|{\mathcal T}_{n-k})
&=\EE_{\mu}\left(\sum_{v \in V_{n}} {\bf 1}_{\{\tau_{n}=v\}} \, y_n{(v)}\mid{\mathcal T}_{n-k}\right)
=\sum_{v\in V_{n}}\EE_{\mu}({\bf 1}_{\{\tau_{n}=v\}}|{\mathcal T}_{n-k})  \, y_{n}(v)\\
&=\sum_{v\in V_{n}}\sum_{u\in V_{n-k}} {\bf 1}_{\{\tau_{n-k}=u\}} \frac{\int_{\{\tau_{n-k}=u\}} {\bf 1}_{\{\tau_{n}=v \}} d\mu}{\mu(\tau_{n-k}=u)} \, y_{n}(v)\\
&=\sum_{u\in V_{n-k}} {\bf 1}_{\{\tau_{n-k}=u\}} \sum_{v\in V_{n}} \mu(\tau_{n}=v|\tau_{n-k}=u) y_{n}(v) . \\
\end{align*}

We have $\EE_{\mu}(Y_{n})=\EE_{\mu}(X_{n})=\EE_{\mu}(\btheta_{n}-\EE_{\mu}(\btheta_{n}))=0$. So, 
$\sum_{v\in V_{n}} \mu(\tau_{n}=v) y_{n}(v)=0$ and thus
$$\EE_{\mu}(Y_{n}|{\mathcal T}_{n-k})=\sum_{u\in V_{n-k}} {\bf 1}_{\{\tau_{n-k}=u\}} 
 \sum_{v\in V_{n}} (\mu(\tau_{n}=v|\tau_{n-k}=u) -\mu(\tau_{n}=v)) y_{n}(v)  .$$
From this expression and Lemma \ref{lem:geometricdeacrieasing} we get
\begin{align*}
|\EE_{\mu}(Y_{n}|{\mathcal T}_{n-k})| 
&\leq  \sum_{u\in V_{n-k}} \sum_{v\in V_{n}} |\mu(\tau_{n}=v|\tau_{n-k}=u) -\mu(\tau_{n}=v)| |y_{n}(v)|  
\leq  K' \zeta^{k} \xi_{n},
\end{align*} 
where $K$ is a universal constant and we have used that $|Y_{n}|\leq K \xi_{n}$.
From here we deduce that 
\begin{align*}
|\EE_{\mu}(Y_{n} Y_{n-k})| = &
|\EE_{\mu}(\EE_{\mu}(Y_{n} Y_{n-k}|{\mathcal T}_{n-k}))| 
=|\EE_{\mu}( Y_{n-k} \EE_{\mu}(Y_{n} |{\mathcal T}_{n-k}))| \\
\leq & K \zeta^{k} \xi_{n} K \xi_{n-k}=K \zeta^{k} \xi_{n} \xi_{n-k}.
\end{align*}

From previous discussions and formulas we get,
\begin{align*}
\langle \sum_{i=m}^{n} Y_{i}, \sum_{j=m}^{n} Y_{j}\rangle 
&= \sum_{i=m}^{n} \sum_{j=m}^{n} \langle Y_{i}, Y_{j} \rangle 
\leq K \sum_{i=m}^{n} \sum_{j=m}^{n} \zeta^{|i-j|} \xi_{i} \xi_{j} 
\leq K \sum_{k=0}^{n-m} \zeta^{k} \sum_{l=m}^{n-k} \xi_{l} \xi_{l+k} \\
&\leq K \sum_{k=0}^{n-m} \zeta^{k} \sum_{l=m}^{n} \xi_{l}^{2} 
\leq K \frac{\zeta^{n-m+1}-1}{\zeta-1} \sum_{l=m}^{n} \xi_{l}^{2} . \\
\end{align*}
This implies that $\EE_{\mu}((\sum_{l=m}^{n}Y_{l})^{2})\leq K \sum_{l=m}^{n} \xi_{l}^{2}$. So 
$\sum_{n\geq 1} Y_{n}$ converges in  $L^2 (X,\mu)$.
\medskip 

Finally, the conclusion from the last two computations is that  $\sum_{n\geq 1} X_{n}$ converges in  $L^2 (X,\mu)$ too.

\subsubsection{End of the proof of Theorem \ref{theo:cnsvpmedible}: construction of an eigenfunction for $\lambda$} From previous discussion we get that 
$f_n=\exp(2i\pi  \sum_{k=1}^{n-1} X_k)$ converges in $L^2 (X,\mu)$ to some function $f$.
For $n\geq 1$ and $x \in X$ we have 
\begin{align*}
&f_n(Tx)/f_n(x) \\ 
=&\exp 
\left(
2i\pi  
\sum_{k=1}^{n-1} \theta_k(s_k(Tx),\tau_k(Tx),\tau_{k+1}(Tx)) -\theta_k(s_k(x),\tau_k(x),\tau_{k+1}(x))  
\right) \\
=&\exp 
\left(
2i\pi \alpha \sum_{k=1}^{n-1} (\langle s_k(Tx),h_k \rangle -  \langle s_k(x),h_k \rangle  
\right.\\ 
& \left.
- 2i\pi \alpha \sum_{k=1}^{n-1} (\rho_{k+1}(\tau_{k+1}(Tx))-\rho_{k+1}(\tau_{k+1}(x))-\rho_k(\tau_k(Tx))+\rho_k(\tau_k(x)) ) 
\right) \\
= &\exp 
( 
2i\pi \alpha \sum_{k=1}^{n-1} (\langle s_k(Tx),h_k \rangle - 
\langle s_k(x),h_k \rangle  \\
& - 2i\pi \alpha ( \rho_{n}(\tau_{n}(Tx))-\rho_{n}(\tau_{n}(x))-\rho_1(\tau_1(Tx))+\rho_1(\tau_1(x))) ) \\
=&\exp (2i\pi \alpha ( r_n(Tx)-r_n(x) -\rho_{n}(\tau_{n}(Tx))+\rho_{n}(\tau_{n}(x)))), 
\end{align*}
where in the last equality we have assumed without loss of generality that $\rho_1$ is a constant function. 

Now, if $x$ is not in the base $\bigcup_{v\in V_n} B_n(v)$ of level $n$, then 
$r_n(Tx)-r_n(x)=-1$ and $\tau_{n}(Tx)=\tau_{n}(x)$. 
Thus, in this case, $f_n(Tx)/f_n(x)=\lambda^{-1}$. 
Since $\lim_{n\to +\infty}$ $\mu(\bigcup_{v\in V_n} B_n(v))=0$ we conclude that 
$f\circ T=\lambda^{-1} f$ in $L^2 (X,\mu)$. This proves that condition \eqref{eq:cnsvpmedible} is a sufficient condition for $\lambda$ to be an eigenvalue for $\mu$.
$\square$

\section{Examples and Applications} 
\label{examples}

In this section we illustrate the use of the main theorems of this article presenting two examples and two applications. These examples and applications show firstly how we can handle the combinatorial objects that appear in the formulation of the main theorems, and secondly they show how these theorems can solve some precise questions in the theory of minimal Cantor systems concerning eigenvalues. In particular, questions that relate the study of eigenvalues with strong orbit equivalence and dimension groups theory.

\subsection{Eigenvalues of minimal Cantor systems of Toeplitz-type}\label{sub sub:toeplitz}

A minimal Cantor system $(X,T)$ is said to be of \emph{Toeplitz-type} if it has a Bratteli-Vershik representation 
$B=\left( V,E,\preceq \right)$ satisfying the \emph{equal path number} property. That is, the number of edges ending at some fixed vertex is constant within the respective level: for any $n\geq 1$, $\#\set{e\in E_n\talque \range(e)=u}$ is a positive integer independent of $u\in V_{n}$. This integer will be denoted by $q_n$ and $(q_n\talque n\geq 1)$ is called \emph{the characteristic sequence} of the system. We also define the quantities $p_n=q_1q_2\cdots q_n$ for $n\geq 1$ and $q_{m,n}=q_{m+1}q_{m+2}\cdots q_{n}$ for $1 \leq m < n$. We notice that for $1 \leq m < n$, 
$$
h_m(u)/h_n(v)=1/q_{m,n} \; \text{  for all  } u\in V_m \text{  and  } v\in V_n.
$$

It is easy to show that odometers (\emph{i.e}, equicontinuous minimal Cantor systems) have a Bratteli-Vershik representation of Toeplitz-type. Also, every Toeplitz subshift can be represented by a Bratteli-Vershik system of Toeplitz-type \cite{gjtoeplitz}.

Let $(X,T)$ be a minimal Cantor system of Toeplitz-type and finite rank $d$ given by the Bratteli-Vershik representation $B=\left( V,E,\preceq \right)$ satisfying the equal path number property. As in  \cite{rangofinito} and \cite{dfm}, in this case we let $\overline{s}_{m,n}(x)$ stand for $\la s_{m,n}(x), h_1\ra$ for $x\in X$ and $1 \leq m < n$. We have $\la s_{m,n}(x), h_n\ra= p_n\overline{s}_{m,n}(x)$, and that for each $n\geq 1$ the function $\overline{s}_{n}=\overline{s}_{n,n+1}$ takes all the values between $0$ and $(q_{n+1}-1)$. We also define the set $\overline{S}_{m,n}(u,v)=\set{\overline{s}_{m,n}(x)\talque x\in X, \tau_m(x)=u, \tau_{n}(x)=v}$.

\subsubsection{Continuous eigenvalues of $(X,T)$}
The following characterization for continuous eigenvalues of a Toeplitz subshift is well-known (see \cite{Keane,Williams} or Theorem 25 in \cite{rangofinito}). Here we provide a proof in the context of minimal Cantor systems of Toeplitz-type to emphasize the use of Theorem \ref{theo:nec-suff-cont}. 

\begin{theo}
The complex number $\exp(2 i \pi \alpha)$ is a continuous eigenvalue of $(X,T)$ if and only if $\alpha = a/p_N$, for some $a\in\ZZ$ and $N\geq 1$.
\end{theo}
\begin{proof}
Let $\exp(2i \pi \alpha)$ be a continuous eigenvalue of $(X,T)$ with $\alpha\in [0,1)$. Then, by \eqref{eq:eta}, for any $n\geq 1$, $\alpha h_n= \alpha p_n (1,\ldots,1)=\eta_n + \nu_n$
with $\eta_n \to 0$ as $n\to + \infty$ and $\nu_n$ an integer vector. As remarked before, from \eqref{eq:eta1} we deduce that $ \eta_n \cdot \mu_n =0$ for any large integer $n$. 

Now, since for all large integer $n$ we have $h_n=p_n (1,\ldots,1)$, then $\eta_n= \bar\eta_n(1,\ldots,1)$ and $\nu_n=\bar\nu_n (1,\ldots,1)$ for some real numbers $\bar\eta_n$ and $\bar\nu_n$. Hence, 
$ \eta_n \cdot \mu_n =\bar\eta_n  (1,\ldots,1) \cdot \mu_n =0$. This implies that $\bar\eta_n=0$ and thus $\alpha= \bar\nu_n/p_n$ for all large enough integer $n$. So $\alpha$ has the desired form.  

On the other hand, if $\alpha=a/p_N$ for some $a\in\ZZ$ and $N\geq 1$, then $\vvert \alpha p_n\overline{s}_n(x)\vvert=0$ for all $n>N$ and $x\in X$, and therefore condition (2) of Theorem \ref{theo:nec-suff-cont} holds. 
\end{proof}
 
\subsubsection{Non continuous eigenvalues  of $(X,T)$}
Let $\mu$ be an ergodic measure of $(X,T)$. 
Using the notation established for minimal Cantor systems of Toeplitz-type, from Theorem \ref{theo:cns_sinrho} we get the following result proved in 
\cite[Corollary 5]{dfm}.

\begin{theo}\label{theo:cns_toeplitz_new}
The complex number $\lambda=\exp(2i\pi \alpha)$ is an eigenvalue of $(X,T)$ for $\mu$ if and only if
\begin{equation}\label{eq:cns_sinrho_toeplitz}
\sum_{u\in I_{\mu}}\frac{1}{q_{m,n}}\Bigg|\sum_{s\in\overline{S}_{m,n}(u,v)}\!\!\!\lambda^{p_ms}\Bigg|\xrightarrow[m\to+\infty]{} 1, \; \text{for all } v\in I_{\mu},
\end{equation}
uniformly for $n>m$. 
If in addition $\alpha = a/b$, with $(a,b)=1$ and $b/(b,p_n)>1$ for all large enough $n$, then $\lambda$ is a non continuous eigenvalue. 
\end{theo}

From \cite[\S 7]{rangofinito} we know that an eigenvalue $\lambda=\exp(2i\pi \alpha)$ of $(X,T)$ for $\mu$ is necessarily rational, {\it i.e.}, $\alpha=a/b$ with $(a,b)=1$. So the condition of last theorem only needs to be verified for rational angles. Also, 
it is interesting to notice that the same condition implies that 
$b/(b,p_n)\leq d$ for all large enough $n$ (this follows from the proof of Lemma 13(1) of \cite{dfm}), which limits the possibility of having non continuos eigenvalues. 

\subsection{Continuous eigenvalues and strong orbit equivalence}
\label{sec:modifying}

As can be seen in \cite{lr} and \cite{necesariasuficiente}, in the linearly recurrent case the local orders of Bratteli-Vershik representations do not play any role in the existence of continuous and measurable eigenvalues. In other words, if we take two Cantor minimal systems and their respective linearly recurrent Bratteli-Vershik representations only differ in their local orders, then both systems have the same continuous and non continuous eigenvalues. This is also the case for continuous eigenvalues that are roots of unity in any minimal Cantor system (see Corollary \ref{coro:rootsofunity}), {\it i.e.}, the group of them do not change if we modify the local orders of a given Bratteli-Vershik representation. 

The case of irrational continuous eigenvalues ({\it i.e.}, continuous eigenvalues that are not roots of unity) is more complicated and here we use Theorem \ref{theo:nec-suff-cont} to show how some modifications of local orders can either leave invariant or modify significantly the group of continuous (irrational) eigenvalues. First we propose a type of controlled modifications of the local orders which do not alter the group of continuous eigenvalues. Then we show that it is possible to change the local orders of a proper Bratteli-Vershik representation of a minimal Cantor system to produce a strong orbit equivalent system without irrational continuous eigenvalues which keeps all measure theoretical eigenvalues for any ergodic measure of the system.

\smallskip

Let $(\omega_n;n\geq 2)$ be a sequence of positive integers. A $(\omega_n;n\geq 2)$-order modification of a Bratteli-Vershik representation $B=\left( V,E,\preceq \right)$ is a new Bratteli-Vershik representation $B'=\left( V,E,\preceq ' \right)$ (only local orders are modified) such that for all $n\geq 2$ and $u\in V_n$ the local orders $\preceq$ and $\preceq'$ at edges with range $u$ differ at most $w_n$ times. We say this modification is proper if $B'$ is proper. 

\begin{coro}
\label{coro:necmodifying}
Let $(X,T)$ be a minimal Cantor system given
by a proper Bratteli-Vershik representation $B=\left( V,E,\preceq \right)$ and let
$\lambda=\exp(2i\pi \alpha)$ be a continuous eigenvalue of $(X,T)$. Let
$(\omega_n;n\geq 2)$ be a sequence of positive integers such that $\sum_{n\geq 2} \omega_{n+1} \vvert \alpha h_n \vvert < +\infty$.
Then, $\lambda$ is a continuous eigenvalue of any minimal Cantor system represented by a proper $(\omega_n;n\geq 2)$-order modification of $B$.
\end{coro}

\begin{proof}
Let $B'$ be a proper $(\omega_n;n\geq 2)$-order modification of $B$.
For each $n\geq 1$ and $x \in X$, let $s_n(x), s'_n(x)\in  \NN^{V_n}$ be the suffix vectors of $x \in X$ associated to diagrams $B$ and $B'$ respectively. Then, $s_n(x) = s'_n(x) + \Delta_{n+1}(x)$, where the error term satisfies $|\Delta_{n+1}(x)| \leq \omega_{n+1}$. Moreover, 
\begin{align*}
\vvert\langle  s'_n(x) ,\alpha h_n  \rangle\vvert & \leq \vvert\langle  s_n(x) ,\alpha h_n  \rangle \vvert +  \vvert\langle  \Delta_{n+1} ,\alpha h_n  \rangle \vvert 
\leq \vvert\langle  s_n(x) ,\alpha h_n  \rangle \vvert +  \omega_{n+1}\vvert  \alpha h_n  \vvert  .
\end{align*}

By Theorem \ref{theo:nec-suff-cont} one has that $\sum_{n\ge 1}  \max_{x\in X} \vvert\langle  s_n(x) ,\alpha h_n  \rangle\vvert  <+\infty$ and by hypothesis $\sum_{n\geq 2} \omega_{n+1} \vvert \alpha h_n \vvert <+\infty$, then $\sum_{n\ge 1}  \max_{x\in X} \vvert\langle  s'_n(x) ,\alpha h_n  \rangle\vvert  <+\infty$. The result follows from condition (2) of Theorem \ref{theo:nec-suff-cont}.
\end{proof}

Before continuing the discussion let us recall the notions of orbit and strong orbit equivalence. Two dynamical systems $(X,T)$ and $(\tilde X,\tilde T)$ are {\em orbit equivalent} if there exists a homeomorphism $\phi : X \to \tilde X$ sending orbits to orbits: for all $x\in X$, 
$$
\phi \left(\{ T^n x ;  n\in \mathbb{Z} \} \right) = \{ \tilde T^n \phi (x) ;  n\in \mathbb{Z} \}.
$$
This induces the existence of maps $\vartheta : X\to \mathbb{Z}$ and $\kappa : X \to \mathbb{Z}$ satisfying for all $x\in X$,
$
\phi \circ T (x) = \tilde T^{\vartheta (x)}\circ \phi (x) \hbox{ and } \phi \circ T^{\kappa (x)} (x) = \tilde T\circ \phi (x).
$

If there exists $\phi$ as above such that its associated maps $\vartheta$ and $\kappa$ have both at most one point of discontinuity we say that $(X,T)$ and $(\tilde X,\tilde T)$ are {\em strong orbit equivalent}. If $(X,T)$ and $(\tilde X,\tilde T)$ are orbit equivalent (and in particular strong orbit equivalent), there exists an affine isomorphism between their sets of invariant probability measures (see \cite[Theorem~2.2]{gps}). If $\mu$ is an invariant probability measure of $(X,T)$ we call $\tilde\mu$ the corresponding invariant probability measure of $(\tilde X,\tilde T)$ given by this isomorphism.

From the viewpoint of Bratteli-Vershik representations, two Cantor minimal systems $(X,T)$ and $(\tilde X,\tilde T)$ are strong orbit equivalent if and only if they have corresponding Bratteli-Vershik representations $B = (V,E,\preceq)$ and $\tilde B = (\tilde V,\tilde E,\tilde \preceq)$ each of them differing only in its local orders with some (maybe different in each case) telescoping of a third Bratteli-Vershik representation $\hat B = (\hat V,\hat E,\hat \preceq)$ (see \cite[Theorem~1.1]{GWorbitequiv}). In particular two proper Bratteli-Vershik representations differing only in their local orders are strong orbit equivalent.

In the context of our discussion, it is known (see for example \cite[Theorem~2.2]{ormes}) that continuous eigenvalues that are roots of unity are preserved by strong orbit equivalence (an alternative proof of this is obtained by a direct application of Corollary \ref{coro:rootsofunity}). This is not the case for irrational continuous eigenvalues. In the next result we use Theorem \ref{theo:nec-suff-cont} to prove that the local orders of a proper Bratteli-Vershik representation of a minimal Cantor system can be modified in order to lose all their irrational continuous eigenvalues. At the same time, the resulting strong orbit equivalent system preserves the measure theoretical eigenvalues of the original one. This result complements the representation result of N. Ormes \cite[Theorem~6.1]{ormes} that is used to prove that strong orbit equivalence of minimal Cantor systems is compatible with any group of eigenvalues as long as they have the same continuous eigenvalues that are roots of unity. For some deeper discussions and recent results on the relation of continuous eigenvalues with orbit equivalence see \cite{Cortez&Durand&Petite:2016,GHH}.  

\begin{coro}
\label{coro:erasingcontinuous}
Let $(X,T)$ be a minimal Cantor system given by a proper Bratteli-Vershik representation $B=\left( V,E,\preceq \right)$. There exists a minimal Cantor system $(\tilde X,\tilde T)$ obtained by contractions and modifications of the local orders of $B$ (so it is strong orbit equivalent with $(X,T)$) such that
\begin{itemize}
\item[(1)] it has not irrational continuous eigenvalues,
\item[(2)] for every ergodic probability measure $\mu$ of $(X,T)$, the systems $(X,T, \mu)$ and $(\tilde X,\tilde T, \tilde \mu)$ have the same measure theoretical eigenvalues.
\end{itemize}
In particular, if the original system has no roots of unity (except the trivial one) as continuous eigenvalues, then the resulting system is topologically weakly mixing.
\end{coro}
The last statement involving systems with no (non trivial) roots of unity as eigenvalues is also a consequence of \cite[Theorem~4.4]{GHH} whose proof also consists in making precise contractions and choices for the local orders of a Bratteli diagram. From \cite[Theorem 6.1]{ormes} it can be obtained a similar result: a Cantor minimal system with no non trivial roots of unity as eigenvalues has a strong orbit equivalent system which is weak mixing in the measure theoretical way (with respect to to some given measure). As we will see below in the proof of Corollary \ref{coro:erasingcontinuous}, the contractions and modifications on the local orders presented here allow us to start with any group of continuous eigenvalues and keep control over all the measure theoretical ones.

\begin{proof}
We identify $(X,T)$ with the system given by a proper Bratteli-Vershik representation $B=\left( V,E,\preceq \right)$.
Recall that for all $n\geq 1$ we write $V_n=\{1,\ldots, d_n\}$ and we assume that all maximal edges for the local orders start at vertex $d_n$.  
\smallskip

\noindent {\it 1. Defining a strong orbit equivalent system:}
This kind of construction appears in the context of tiling systems in \cite{sadunpriebe}.

Assume we have a strictly increasing sequence of non negative integers $(\ell_n; n\geq 0)$ with $l_0=0$ and $l_1=1$. 
We use it to define the minimal Cantor system $(\tilde X,\tilde T)$ by telescoping $B$ at levels 
$(\ell_{2n-1}; n\geq 1)$ and then by changing the local orders of the derived system as described in the next paragraph. We call $\tilde B=(\tilde V,\tilde E,\tilde \preceq)$ the resulting diagram. 
The first level of $\tilde B$ coincides with the one of $B$ and for all $n\geq 2$ we have $\tilde V_n=V_{\ell_{2n-1}}$ and $\tilde E_n=E_{\ell_{2n-3},\ell_{2n-1}}$. To make the difference with the diagram $B$ all combinatorial objects associated to the diagram $\tilde B$ will be marked with a tilde. For instance: $\tilde M_n$, $\tilde h_n$, $\tilde d_n$, $\tilde r_n(\cdot)$, $\tilde s_{m,n}(\cdot)$, $\tilde S_{m,n}(\cdot,\cdot)$, etc. 

For all $n\geq 2$ we modify the local orders induced on $\tilde E_{n}=E_{\ell_{2n-3},\ell_{2n-1}}$ after telescoping as follows. For all $v \in \tilde V_n$ let $e_v \in E_{\ell_{2n-2},\ell_{2n-1}}$ be the maximal path in $B$ going from $d_{\ell_{2n-2}}$ to $v$. We reorder only those edges in $\tilde E_n$ finishing in vertex $v$ that come from a path in $B$ having $e_v$ as a sub-path. We do this from left to right: the new order puts first the paths starting in vertex $1\in \tilde V_{n-1}$ and then those starting in vertex $2\in \tilde V_{n-1}$, etc., until arriving to edges starting in vertex $\tilde d_{n-1}=d_{\ell_{2n-3}}$. We keep the local orders induced by the telescoping in the remaining edges of $\tilde E_{n}$. 

Call $(\tilde X,\tilde T)$ the minimal Cantor system associated to $\tilde B$. By construction, $(\tilde X,\tilde T)$ is strong orbit equivalent with $(X,T)$. 
\smallskip

\noindent {\it 2. Candidates to be irrational continuous eigenvalues:}
Fix an ergodic probability measure $\mu^{(0)}$ of $(X,T)$. Consider the countable set $\mathcal{A}$ of irrational real numbers $\alpha$ such that there exist an integer $N\geq 1$ and an integer (row) vector $\nu \in \ZZ^{V_N}$ such that $\alpha=\nu \cdot \mu^{(0)}_N$ and that for all $m\geq 1$ they have the decomposition $\alpha h_m=\eta_m+\nu_m$, with $\eta_m$ a real vector and $\nu_m$ an integer vector such that the sequence $(\eta_m;m\geq 1)$ converges to $0$ on a subsequence (but it is never equal to $0$ due to the irrationality of $\alpha$). By \eqref{eq:eta} and \eqref{eq:eta2}, $\mathcal{A}$ contains all real numbers $\alpha$ such that $\lambda=\exp(2i\pi \alpha)$ is an irrational continuous eigenvalue of $(X,T)$. Moreover, since the definition of the set $\mathcal{A}$ is independent of the local orders of the Bratteli-Vershik representation $B$, then $\mathcal{A}$ also contains all  
the real numbers $\alpha$ such that $\lambda=\exp(2i\pi \alpha)$ is an irrational continuous eigenvalue of any minimal Cantor system derived from $B$ by a telescoping followed by a modification of the derived local orders. 

Let $(\alpha_n;n\geq 1)$ be a sequence in $\mathcal{A}$ such that each element of $\mathcal{A}$ appears infinitely many times. Put $\lambda_n=\exp(2i\pi\alpha_n)$.
By definition of $\mathcal{A}$, for every $m\geq 1$ we can write 
$\alpha_n h_m=\eta^{(n)}_m+\nu^{(n)}_m$ with $\eta^{(n)}_m$ a real vector and $\nu^{(n)}_m$ an integer vector such that the sequence $(\eta^{(n)}_m;m\geq 1)$ converges to $0$ on a subsequence and is never equal to $0$.
\smallskip

\noindent {\it 3. Fixing the sequence $(\ell_n; n\geq 0)$:}
We define recursively the sequence $(\ell_n; n\geq 0)$ satisfying the following conditions:
\begin{enumerate}[(a)]
\item\label{eq:firstlevels} $\ell_0=0$ and $\ell_1=1$;  

\item\label{eq:oddlevels1} $h_{\ell_{2n-1}}(v) \geq n$ for all $n\geq 2$ and $v \in V_{\ell_{2n-1}}$;

\item\label{eq:oddlevels2} $\matrizp{\ell_{2n-2}}{\ell_{2n-1}}{d_{\ell_{2n-2}}}{v} > n^2$ for all $n\geq 2$ and $v\in V_{\ell_{2n-1}}$;

\item\label{eq:oddlevels3} $|\eta^{(n+1)}_{\ell_{2n-1}}| < 1/4$ for all $n\geq 2$;

\item\label{eq:evenlevels} for all $n\geq 2$,
$$
\frac{1}{\left\vert \eta^{(n)}_{\ell_{2n-3}}(d_{\ell_{2n-3}})\right\vert} < P_{\ell_{2n-3},\ell_{2n-2}}(d_{\ell_{2n-3}},d_{\ell_{2n-2}})-1.
$$
\end{enumerate}
Notice that condition \eqref{eq:evenlevels} is possible since we do not allow $\mathcal{A}$ to have rational numbers.

Let $n\geq 2$ and assume that $\ell_0,\ldots,\ell_{2n-3}$ are already defined and satisfies conditions \eqref{eq:firstlevels}-\eqref{eq:evenlevels}. First we take $\ell_{2n-2}$  enough large such that 
\eqref{eq:evenlevels} holds. Now that we have defined $\ell_{2n-2}$ it is enough to take $\ell_{2n-1}$ enough large to satisfy conditions \eqref{eq:oddlevels1}, \eqref{eq:oddlevels2} and \eqref{eq:oddlevels3}. This last property can be achieved by the choice of $\alpha_{n+1}$. 
\smallskip

\noindent {\it 4. $(\tilde X,\tilde T)$ satisfies condition (2):}

Let us consider an ergodic probability measure $\mu$ of $(X,T)$. For $n\geq 2$ let $D_{n}$ be the set of points in $\tilde X$ passing through edges of $\tilde E_n$ which were derived from paths in $E_{\ell_{2n-3},\ell_{2n-1}}$ containing a maximal sub-path $e_v$ for some $v \in V_{\ell_{2n-1}}$. We have 
$$\tilde\mu(D_{n})=\sum_{v \in V_{\ell_{2n-1}}} \mu_{\ell_{2n-1}}(v) \ h_{\ell_{2n-2}}(d_{\ell_{2n-2}}).$$  
Then, by \eqref{eq:measure} and condition \eqref{eq:oddlevels2} above, we get
\begin{align*}
1 & \geq h_{\ell_{2n-2}}(d_{\ell_{2n-2}})\mu_{\ell_{2n-2}}(d_{\ell_{2n-2}}) \\
 &= h_{\ell_{2n-2}}(d_{\ell_{2n-2}}) \sum_{v \in V_{\ell_{2n-1}}} \matrizp{\ell_{2n-2}}{\ell_{2n-1}}{d_{\ell_{2n-2}}}{v} \mu_{\ell_{2n-1}}(v) > \tilde\mu\left(D_n\right)n^2. 
\end{align*}
Therefore, $\tilde\mu\left(D_n\right)<1/n^2$ and by Borel-Cantelli $\tilde\mu(\displaystyle\limsup_{n\to +\infty}D_{n})=0$.

Let $\displaystyle \tilde x\not \in \limsup_{n\to +\infty}D_{n}$. Then, from a level $n(\tilde x) \geq 2$, $\tilde x$ does not pass by any edge of $\tilde E_n$ which comes from a path in $E_{\ell_{2n-3},\ell_{2n-1}}$ containing a maximal sub-path $e_v \in E_{\ell_{2n-2, 2n-1}}$. This implies that, if we identify $B$ with its telescoping with respect to levels $(\ell_{2n-1};n\geq 1)$, the suffix vectors associated to $\tilde x$ in $\tilde B$ differ in finitely many levels with those associated to $\tilde x$ when seen as a point in $B$. Thus, keeping this identification for $B$, we get that for all large enough values of $m$ we have $\tilde r_m(\tilde x) = r_m(\tilde x) + c(\tilde x)$, where $c(\tilde x)$ is a constant depending only on $\tilde x$. By Theorem \ref{theo:cns_old_rho} we conclude that eigenvalues of $(X,T)$ for $\mu$ coincide with eigenvalues of $(\tilde X,\tilde T)$ for $\tilde \mu$. 
\smallskip

\noindent {\it 5. The $\lambda_n$'s are not continuous eigenvalues of $(\tilde X,\tilde T)$:}
We will use Theorem \ref{theo:nec-suff-cont} part (3). 

Let $n\geq 3$. A careful inspection of the diagram $\tilde B$ shows that for any $v\in \tilde V_{n}=V_{\ell_{2n-1}}$ we have
$$\set{(0,0,\dots, 0,t)\in \NN^{\tilde V_{n-1}}; t < \matrizp{\ell_{2n-3}}{\ell_{2n-2}}{d_{\ell_{2n-3}}}{d_{\ell_{2n-2}}}} \subseteq \tilde S_{n-1}(\tilde d_{n-1},v).$$
By property \eqref{eq:evenlevels} above we can take $t=\ds \left \lfloor \frac{1}{2 \left\vert\eta^{(n)}_{\ell_{2n-3}}(d_{\ell_{2n-3}})\right\vert}\right \rfloor +1$ in the previous set. Thus, using the fact that $|\eta^{(n)}_{\ell_{2n-3}}| < 1/4$ 
we get  

\begin{align*}
\vvert \alpha_n \left\langle (0,\ldots, 0,t),  \tilde h_{n-1} \right\rangle \vvert &=  
\vvert \alpha_n \left\langle (0,\ldots, 0,t), h_{\ell_{2n-3}} \right\rangle \vvert \\ &=
\vvert t \cdot \eta^{(n)}_{\ell_{2n-3}}(d_{\ell_{2n-3}})  \vvert \in \left]1/4,1/2\right].
\end{align*}
Since each value $\alpha$ in the sequence $(\alpha_n;n\geq 1)$ appears infinitely many times, then the series 
$$\sum_{n\geq 1} \max_{s \in \tilde S_n(\tilde d_n,\tilde d_{n+1})} \vvert \alpha \langle s,\tilde h_n \rangle \vvert$$
in Theorem \ref{theo:nec-suff-cont} part (3) cannot converge. Thus $\lambda_n$ is not a continuous eigenvalue of $(\tilde X, \tilde T)$ for all $n\geq 1$. 
\end{proof}

A natural question arising from last corollary is the following. Given a minimal Cantor system $(X,T)$, is it possible to realise any subgroup of its group of continuous eigenvalues as the group of continuous eigenvalues of some minimal Cantor system $(\tilde X,\tilde T)$ which is strong orbit equivalent with 
$(X,T)$ ?

Proposition 25 in \cite{Cortez&Durand&Petite:2016} (see also \cite{Itza-Ortiz:2007}) asserts in specific cases that it is possible to realise some subgroups, but it is not clear which ones.
Moreover, there are strong obstructions for this kind of realisations. For example, it is well-known that 
$G = \{ \exp (2i\pi  \beta ); \beta \in \mathbb{Z} +\alpha \mathbb{Z}  \}$ is the group of continuous eigenvalues of a Sturmian subshift, where $\alpha$ is the angle of the rotation associated to the Sturmian system. The main result in \cite{Cortez&Durand&Petite:2016} shows that the only other subgroup of $G$ that can be realised within the strong orbit equivalence class of the Sturmian system is the other trivial subgroup. 

We have discussed the relations between the group of continuous eigenvalues and the notion of strong orbit equivalence. In the case of non continuous eigenvalues the situation is completely different. Counterexamples can be obtained using \cite[Corollary~6.2]{ormes}, a remarkable generalization of Dye's theorem \cite{Dye} involving strong orbit equivalence. From Orme'{}s result it can be proved the existence of strong orbit equivalent minimal Cantor systems, each one of them with any prescribed group of eigenvalues for a given ergodic probability measure. Thus non continuous eigenvalues not need to be preserved under strong orbit equivalence, even in the case of rational non continuous ones.  For concrete counterexamples in the case of rational non continuous eigenvalues we can use the Toeplitz systems of finite rank in \cite[\S 6]{dfm}. Various modifications on the local orders of each example there lead to a strong orbit equivalent system keeping its topological rank but losing either some or all rational non continuous eigenvalues. The main reason of this comes from the strict restrictions imposed by rational eigenvalues on local orders in the finite rank Toeplitz case (see for example Theorem \ref{theo:cns_toeplitz_new} and \cite[Corollary~4]{dfm}).

\subsection{A topologically weakly mixing minimal Cantor system of rank two admitting all rational numbers as non continuous eigenvalues}
We have proved that in minimal Cantor systems of Toeplitz-type the finite rank condition strongly restricts the existence of non continuous rational eigenvalues. In this section we prove that for general minimal Cantor systems this is not longer true by constructing a topologically weakly mixing minimal Cantor system of rank two admitting all rational numbers as non continuous eigenvalues.

Set a sequence of positive integers $\left(b_n\talque n>1\right)$ such that $\sum_{n>1} 1/b_n<\infty$. Consider the minimal Cantor system $(X,T)$ given by the proper Bratteli-Vershik representation of rank $2$, where $V_n=\left\{1,2\right\}$ for all $n\geq 1$, $h_1=(1,1)$ and the rest of its edges and local orders are described as follows. For each $n \geq 1$ consider the function $\theta_{n+1}: V_{n+1} \to V_{n}^{*}$ such that for each $u\in V_{n+1}$ its image $\theta_{n+1}(u)$ is the word in $V_{n}$ which lists, following the local order, the sources of all edges in $E_{n+1}$ ending at $u$: 

\begin{align*}
\begin{array}{rcl}
\theta_2(1) &= & 2(1)^52 \\
\theta_2(2) &= & 2(1)^32
\end{array}
& \textrm{\; \; \; \; and\; \; \; \; }
\begin{array}{rcl}
\theta_{n+1}(1) &=& 211(21)^{b_n+1}12 \\
\theta_{n+1}(2) &=& (21)^{b_n}2
\end{array}
 \; \; \textrm{for\, } n > 1,
\end{align*}
\begin{prop}\label{prop:ex_uniquely_ergodic}
The minimal Cantor system $(X,T)$ is uniquely ergodic.
\end{prop}
\begin{proof}
Let $\mu$ be an ergodic measure for $(X,T)$. The following matrix relations hold in the diagram defining 
$(X,T)$. For all $n>1$, 
\begin{align}
\left(h_{n+1}(1), h_{n+1}(2)\right) & =
\left(h_{n}(1),h_{n}(2)\right)
\left(
\begin{array}{cc}
b_n+4 & b_n\\
b_n+3 & b_n+1
\end{array}
\right), \label{eq:ex_basic_relations_h}\\
\left(
\begin{array}{c}
\mu_{n}(1) \\ \mu_{n}(2)
\end{array}
\right) & =
\left(
\begin{array}{cc}
b_n+4 & b_n\\
b_n + 3& b_n + 1
\end{array}
\right)
\left(
\begin{array}{c}
\mu_{n+1}(1) \\ \mu_{n+1}(2)
\end{array}
\right). \label{eq:ex_basic_relations_mu}
\end{align}
Using these relations one has that for all $n>2$
\begin{align*}
\frac{h_n(1)}{h_n(1)+h_n(2)}=\frac{h_{n-1}(1)(b_{n-1}+4)+h_{n-1}(2)(b_{n-1}+3)}{(h_{n-1}(1)+h_{n-1}(2))(2b_{n-1}+4)},
\end{align*}
and thus 
\begin{align*}
\frac{b_{n-1} + 3}{2b_{n-1}+4} < \frac{h_n(1)}{h_n(1)+h_n(2)} < \frac{b_{n-1}+4}{2b_{n-1}+4}.
\end{align*}
Also,
\begin{align*}
\mu_n(1)&=\frac{(b_n+4)\mu(\tau_{n+1}=1)}{h_{n+1}(1)} + \frac{b_n\mu(\tau_{n+1}=2)}{h_{n+1}(2)}, \\
&=\frac{(b_n+4)\mu(\tau_{n+1}=1)}{(b_{n}+4)h_n(1)+(b_n+3)h_n(2)}+ \frac{b_n\mu(\tau_{n+1}=2)}{b_{n}h_n(1)+(b_n+1)h_n(2)}.
\end{align*}
Since $\mu(\tau_{n+1}=1)+\mu(\tau_{n+1}=2)=1$ we get
\begin{align*}
\frac{b_n}{b_n+4}\left[\frac{1}{h_n(1)+h_n(2)}\right] < \mu_n(1) < \frac{b_n+4}{b_n}\left[\frac{1}{h_n(1)+h_n(2)}\right].
\end{align*}
This implies that $\mu(\tau_n=1)=h_n(1)\mu_n(1) \xrightarrow[n\to+\infty]{} 1/2$ and $\mu(\tau_n=2) \xrightarrow[n\to+\infty]{} 1/2$. Therefore the system is of exact finite rank, which implies its unique ergodicity (see last paragraph of Section \ref{subsec:clean}).
\end{proof}

To compute eigenvalues of $(X,T)$ for its unique invariant measure $\mu$ we will use
Theorem \ref{theo:cns_sinrho}. Notice from the last proof that the diagram defining $(X,T)$ is clean, so all the requirements of the theorem hold. Moreover, $I_\mu=\{1,2\}$. 

\begin{rmrk}\label{rmrk:forma_matrices_ejemplo}
All matrices $M_{n}$, for $n>2$, are of the form $\left(\begin{smallmatrix}a+1&b\\a&b+1\end{smallmatrix}\right)$ with $a,b \in\ZZ$. Then recursively we conclude that matrices $P_{m,n}$, with $2\leq m<n$, are of the same form.
\end{rmrk}

Before applying Theorem \ref{theo:cns_sinrho} we need to understand 
the suffix sets $\sufijocbb{m}{n}{u}{v}$ for $1 \leq m < n$ and $u,v\in \{1,2\}$. We will use some relations of suffix vectors contained in Section 2.2.4, particularly we recall the equality \eqref{eq:formulasuffix_lmn},
\begin{align*}
s_{\ell,n}(x) =  s_{\ell,m}(x) + s_{m,n}(x)P^{T}_{\ell,m}, \textrm{\; for \, } 0\leq \ell < m < n \textrm{ and } x\in X. 
\end{align*}
Also recall that $s_{n}(x)$ stands for $s_{n,n+1}(x)$, with $n\geq 0$ and $x\in X$.
\begin{lemma}\label{lemma:ex_sufijos_error}
Let $L_1, L_2$ and $L_3$ be the following subsets of row vectors of $\ZZ^{2}$: 
\begin{align*}
&L_1=\left\{\left(a,a\right)\, \talque\, a\in\ZZ \right\},
L_2=\left\{\left(a+1,a\right)\, \talque\, a\in\ZZ \right\} \textrm{\; and\; }\\
&L_3=\left\{\left(a,a+1\right)\, \talque\, a\in\ZZ \right\}.
\end{align*}
Then, for all large enough integers $2<m<n$, the following quotients
$$
\frac{\left|\sufijocbb{m}{n}{1}{1}\setminus L_1\right|}{\matrizp{m}{n}{1}{1}}, \, \frac{\left|\sufijocbb{m}{n}{2}{1}\setminus L_2\right|}{\matrizp{m}{n}{2}{1}},\, \frac{\left|\sufijocbb{m}{n}{1}{2}\setminus L_3\right|}{\matrizp{m}{n}{1}{2}}\textrm{\; \; and\; \; }\frac{\left|\sufijocbb{m}{n}{2}{2}\setminus L_1\right|}{\matrizp{m}{n}{2}{2}},
$$
are bounded from above by $\displaystyle\sum_{k=m}^{n-1}2/b_k$.
\end{lemma}
\begin{proof}

Take integers $m> 2$ and $k>0$. First we estimate the cardinality of $\sufijocbb{m}{m+k+1}{1}{1}\setminus L_1$.

Let us take $x\in X$ such that $\torref{m}(x)=1$ and $\torref{m+k+1}(x)=1$, then $s_{m,m+k+1}(x)\in \sufijocbb{m}{m+k+1}{1}{1}$. Using \eqref{eq:formulasuffix_lmn} and Remark \ref{rmrk:forma_matrices_ejemplo} we see that 

\begin{align*}
\left.
\begin{array}{c}
s_{m+k}(x)\in L_1$ \textrm{ and } $s_{m,m+k}(x)\in L_1 \\ \textrm{or} \\
s_{m+k}(x)\in L_2$ \textrm{ and } $s_{m,m+k}(x)\in L_3
\end{array}
\right\} &\implica s_{m,m+k+1}(x)\in L_1.
\end{align*}
From this property, if $s_{m,m+k+1}(x)\not\in L_1$ then we get that one of the following excluding properties holds: 
\begin{enumerate}[(a)]
\item $\torref{m+k}(x)=1, s_{m+k}(x)\in L_1$ and $s_{m,m+k}(x)\not\in L_1$.
\item $\torref{m+k}(x)=1$ and $s_{m+k}(x)\not\in L_1$.
\item $\torref{m+k}(x)=2, s_{m+k}(x)\in L_2$ and $s_{m,m+k}(x)\not\in L_3$.
\item $\torref{m+k}(x)=2$ and $s_{m+k}(x)\not\in L_2$.
\end{enumerate}
Also, looking at $\theta_{m+k+1}(1)$, we have $\left|\sufijocb{m+k}{1}{1}\cap L_1\right|=b_{m+k}+2, \left|\sufijocb{m+k}{1}{1}\setminus L_1\right|=2, \left|\sufijocb{m+k}{1}{2}\cap L_2\right|=b_{m+k}+1$ and $\left|\sufijocb{m+k}{1}{2}\setminus L_2\right|=2$. Then, counting following conditions (a) to (d) gives the upper bound
\begin{align*}
\left|\sufijocbb{m}{m+k+1}{1}{1}\setminus L_1\right| &\leq (b_{m+k}+2)\left|\sufijocbb{m}{m+k}{1}{1}\setminus L_1\right| + 2 \matrizp{m}{m+k}{1}{1}\\
& \quad + (b_{m+k}+1)\left|\sufijocbb{m}{m+k}{1}{2}\setminus L_3\right| + 2\matrizp{m}{m+k}{1}{2}.
\end{align*}

Now, cardinalities of $\sufijocbb{m}{m+k+1}{2}{1}\setminus L_2, \sufijocbb{m}{m+k+1}{1}{2}\setminus L_3$ and $\sufijocbb{m}{m+k+1}{2}{2}\setminus L_1$ can be estimated in a similar way, getting
\begin{align*}
\left|\sufijocbb{m}{m+k+1}{2}{1}\setminus L_2\right| &\leq (b_{m+k}+2)\left|\sufijocbb{m}{m+k}{2}{1}\setminus L_2\right| + 2 \matrizp{m}{m+k}{2}{1}\\
& \quad + (b_{m+k}+1)\left|\sufijocbb{m}{m+k}{2}{2}\setminus L_1\right| + \textstyle 2 \matrizp{m}{m+k}{2}{2}, \\
\left|\sufijocbb{m}{m+k+1}{1}{2}\setminus L_3\right| &\leq (b_{m+k}+1)\left|\sufijocbb{m}{m+k}{1}{2}\setminus L_3\right| + b_{m+k}\left|\sufijocbb{m}{m+k}{1}{1}\setminus L_1\right|, \\
\left|\sufijocbb{m}{m+k+1}{2}{2}\setminus L_1\right| &\leq (b_{m+k}+1)\left|\sufijocbb{m}{m+k}{2}{2}\setminus L_1\right| + b_{m+k}\left|\sufijocbb{m}{m+k}{2}{1}\setminus L_2\right|.
\end{align*}

Let $Q_{m,n}$, for $2<m<n$, denote the maximum of the four quotients in the formulation of the Lemma.
We have, 
\begin{align*}
\frac{\left|\sufijocbb{m}{m+k+1}{1}{1}\setminus L_1\right|}{\matrizp{m}{m+k+1}{1}{1}} &\leq 
\frac{(b_{m+k}+2)\left|\sufijocbb{m}{m+k}{1}{1}\setminus L_1\right|}{\sum_{u=1}^{2}\matrizp{m}{m+k}{1}{u}M_{m+k+1}(u,1)} \\
& \qquad + 
\frac{(b_{m+k}+1)\left|\sufijocbb{m}{m+k}{1}{2}\setminus L_3\right|}{\sum_{u=1}^{2}\matrizp{m}{m+k}{1}{u}M_{m+k+1}(u,1)} \\
& \qquad\qquad + 
\frac{2\sum_{u=1}^{2}\matrizp{m}{m+k}{1}{u}}{\sum_{u=1}^{2}\matrizp{m}{m+k}{1}{u}M_{m+k+1}(u,1)} \\
& \leq \frac{(b_{m+k}+2)\left|\sufijocbb{m}{m+k}{1}{1}\setminus L_1\right|}{(b_{m+k}+3)\sum_{u=1}^{2}\matrizp{m}{m+k}{1}{u}} \\
& \qquad + 
\frac{(b_{m+k}+1)\left|\sufijocbb{m}{m+k}{1}{2}\setminus L_3\right|}{(b_{m+k}+3)\sum_{u=1}^{2}\matrizp{m}{m+k}{1}{u}} \\
& \qquad\qquad + 
\frac{2\sum_{u=1}^{2}\matrizp{m}{m+k}{1}{u}}{(b_{m+k}+3)\sum_{u=1}^{2}\matrizp{m}{m+k}{1}{u}} \\
& \leq \frac{\matrizp{m}{m+k}{1}{1}}{\sum_{u=1}^{2}\matrizp{m}{m+k}{1}{u}}\cdot \frac{\left|\sufijocbb{m}{m+k}{1}{1}\setminus L_1\right|}{\matrizp{m}{m+k}{1}{1}} \\
& \qquad + \frac{\matrizp{m}{m+k}{1}{2}}{\sum_{u=1}^{2}\matrizp{m}{m+k}{1}{u}}\cdot \frac{\left|\sufijocbb{m}{m+k}{1}{2}\setminus L_3\right|}{\matrizp{m}{m+k}{1}{2}} + \frac{2}{b_{m+k}+3} \\
&\leq Q_{m,m+k} + \frac{2}{b_{m+k}}.
\end{align*}
It can be seen that this bound works for the other three remaining quotients. So, we get the recurrence formula
\begin{align*}
Q_{m,m+k+1} \leq Q_{m,m+k} + \frac{2}{b_{m+k}}.
\end{align*}
We conclude by noticing that a direct computation gives $Q_{m,m+1} \leq 2/b_m$.
\end{proof}

\begin{prop}\label{prop:ex_ev_nocontinuos}
For every $\ell\geq 2$, $\lambda= \exp\left(2i\pi/\left(h_{\ell}(1) + h_{\ell}(2)\right)\right)$ is a non continuous eigenvalue of 
$(X,T)$ for the unique invariant measure $\mu$. 
\end{prop}
\begin{proof}
Let us take ${\ell}\geq 2$. By \eqref{eq:ex_basic_relations_h}, it can be seen by induction that for all $m\geq {\ell}$
\begin{align*}
h_m(1)=h_{\ell}(1)\; (\textrm{mod } h_{\ell}(1)+h_{\ell}(2)) \textrm{\; \; and\; \; }
h_m(2)=h_{\ell}(2)\; (\textrm{mod } h_{\ell}(1)+h_{\ell}(2)).
\end{align*}
Then, by Corollary \ref{coro:rootsofunity}, if $\lambda=\exp\left(2i\pi/\left(h_{\ell}(1) + h_{\ell}(2)\right)\right)$ is an eigenvalue then it cannot be continuous.

In order to show that $\lambda$ is actually an eigenvalue, we are going to use Corollary \ref{coro:cs_sin_rho} for vertices $u=1$ and $v=2$, the other cases can be done analogously. Recall that $I_{\mu}=\{1,2\}$.

Let $L_3$ be as in the formulation of Lemma \ref{lemma:ex_sufijos_error} and consider large enough positive integers $n>m\geq {\ell}$. If we take $s\in\sufijocbb{m}{n}{1}{2}\cap L_3$ then there exists $a\in\ZZ$ such that
\begin{align*}
\la s, h_m\ra = \la \left(a,a+1\right), \left(h_m(1), h_m(2)\right) \ra &= a\left(h_m(1) + h_m(2)\right) + h_m(2) \\ &= h_{\ell}(2) \; (\textrm{mod } h_{\ell}(1)+h_{\ell}(2)).
\end{align*} 

Then, for all $s\in \sufijocbb{m}{n}{1}{2}\cap L_3$, $\lambda^{\la s, h_m\ra}=\exp\left(2i\pi h_{\ell}(2)/(h_{\ell}(1)+h_{\ell}(2))\right)$ and
\begin{align*}
\left|\sum_{s\in\sufijocbb{m}{n}{1}{2}}\lambda^{\la s, h_m\ra}\right| 
&\geq \left|\sum_{s\in\sufijocbb{m}{n}{1}{2}\cap L_3}\lambda^{\la s, h_m\ra}\right| - 
\left|\sum_{s\in\sufijocbb{m}{n}{1}{2}\setminus L_3}\lambda^{\la s, h_m\ra}\right|  \\
&\geq \left|\sufijocbb{m}{n}{1}{2}\cap L_3\right| -\left|\sufijocbb{m}{n}{1}{2}\setminus L_3\right| \\
&= \matrizp{m}{n}{1}{2} - 2\left|\sufijocbb{m}{n}{1}{2}\setminus L_3\right|.
\end{align*}

Applying Lemma \ref{lemma:ex_sufijos_error} we obtain
\begin{align*}
1-\sum_{k\geq m} \frac{4}{b_k} \leq \frac{\left|\sum_{s\in\sufijocbb{m}{n}{1}{2}}\lambda^{\la s, h_m\ra}\right|}{\matrizp{m}{n}{1}{2}}\leq 1,
\end{align*}
and we get the condition of Corollary \ref{coro:cs_sin_rho}. This shows that $\exp\left(2i\pi/(h_{\ell}(1)+h_{\ell}(2))\right)$ is a non continuous eigenvalue of $(X,T)$ for $\mu$.
\end{proof}

\begin{rmrk}\label{rmrk:ex_divisores}
If we take $\ell\geq 2$ and an integer $p$ dividing $h_{\ell}(1)+h_{\ell}(2)$, then obviously $\exp(2i\pi/p)$ is an eigenvalue of $(X,T)$ with respect to $\mu$, but it could be a continuous one. 
\end{rmrk}

\begin{coro}\label{coro:ex_primeros_ev}
The complex number $\exp(2i\pi/12)$ is a non continuous eigenvalue of $(X,T)$ for the unique invariant measure $\mu$. Moreover, $\exp(2i\pi/2)$ and $\exp(2i\pi/3)$ are both non continuous eigenvalues for that measure.
\end{coro}
\begin{proof}
As $h_2(1)+h_2(2)=12$ we get from Proposition \ref{prop:ex_ev_nocontinuos} that $\exp(2i\pi/12)$ is a non continuous eigenvalue. Then $\exp(2i\pi/2)$ and $\exp(2i\pi/3)$ are eigenvalues of $(X,T,\mu)$, at least one of them non continuous.

But as in the proof of Proposition \ref{prop:ex_ev_nocontinuos}, for $m\geq 2$
\begin{align*}
h_m(1)=7\; (\textrm{mod } 12) \textrm{\; \; and\; \; }
h_m(2)=5\; (\textrm{mod } 12),
\end{align*}
and from Corollary \ref{coro:rootsofunity} it follows that neither $\exp(2i\pi/2)$ nor $\exp(2i\pi/3)$ can be continuous eigenvalues.
\end{proof}

The example above is indeed a family of systems indexed by the different sequences $\left(b_n\talque n> 1\right)$ satisfying $\sum_{n>1} 1/b_n < \infty$. Choosing some of these sequences we get the following result.

\begin{prop}\label{prop:ex_racionales_vpnocontinuos}
There exists a uniquely ergodic minimal Cantor system $(X,T)$ of topological rank $2$ such that for every $\alpha\in\QQ$, $\exp(2i\pi\alpha)$ is a non continuous eigenvalue of $(X,T)$ for the unique invariant measure $\mu$.
\end{prop}

Before proving Proposition \ref{prop:ex_racionales_vpnocontinuos}, a particular sequence $\left(b_n\talque n> 1\right)$ will be defined recursively.

For $n\geq 1$, let us denote by $p_n$ the $n$-th odd prime number and set $b_2=13$. Notice that $h_2(1) + h_2(2)=12$ and $h_3(1) + h_3(2)=360$, none of them being a multiple of $p_n$ for $n\geq 4$.

Now, let us fix $n\geq 3$ and suppose that we know the elements of $\left(b_n\talque n>1\right)$ up to the $(n-1)$-th one. For $k=2,\ldots, n$, using \eqref{eq:ex_basic_relations_h} we also know the values of $h_k(1)$ and $h_k(2)$ and we will assume that $h_n(1) + h_n(2)$ is not a multiple of $p_{m}$ for $m\geq n+1$.

We choose the element $b_n$ such that
\begin{align}\label{eq:ex_definiendo_sucesion_1ra}
b_n+2 \in \left\{p_1^{\alpha_1}\cdots p_n^{\alpha_n}\in\ZZ \talque \alpha_i>0 \textrm{\, for } i = 1,\ldots, n\right\} \textrm{\, \, and} \\ \label{eq:ex_definiendo_sucesion_2da}
\left(3b_n+8\right)\left(h_n(1) + h_n(2)\right) \neq h_n(2) \; \; \left(\textrm{mod }p_{n+1}\right).
\end{align}
This last equation can be solved because $3\left(h_n(1) + h_n(2)\right)\neq 0 \; \left(\textrm{mod }p_{n+1}\right)$. 

There are infinitely many possibilities for choosing a positive integer $b_n$ satisfying both conditions. This is so because the different possibles $b_n$'{}s not satisfying condition \eqref{eq:ex_definiendo_sucesion_2da} are distant from each other by a multiple of $p_{n+1}$.

Even though $b_{n-1}$ does not appear explicitly in the conditions defining $b_n$, there is a recursion because the heights of level $n$, namely $h_n(1)$ and $h_n(2)$, appear in \eqref{eq:ex_definiendo_sucesion_2da} and they depend on $b_{n-1}$.

In order to complete the recursive step, it only remains to verify that $h_{n+1}(1) + h_{n+1}(2)$ is not a multiple of $p_{m}$ for any $m\geq n+2$. This can be seen from \eqref{eq:ex_basic_relations_h} noticing that
\begin{align}\label{eq:ex_recursion_suma_alturas}
h_{n+1}(1) + h_{n+1}(2) &= 2(b_n + 2)\left(h_n(1) + h_n(2)\right).
\end{align}

Notice that conditions \eqref{eq:ex_definiendo_sucesion_1ra} and \eqref{eq:ex_definiendo_sucesion_2da} are also satisfied for $n=2$.

According to the definition of $\left(b_n\talque n>1\right)$, $\sum_{n>1} 1/b_n<\infty$ because $b_n\geq \Pi_{i=1}^{n}p_i - 2$ for all $n\geq 2$.

\begin{proof}[Proof of Proposition \ref{prop:ex_racionales_vpnocontinuos}]
We consider the system $(X,T)$ of topological rank 2 defined at the beginning of this section along with the sequence $\left(b_n\talque n>1\right)$ defined above.

Let us take a positive integer $p$. From \eqref{eq:ex_definiendo_sucesion_1ra} and applying recursively \eqref{eq:ex_recursion_suma_alturas}, there exists  a sufficiently large integer $m$ such that $p$ divides $h_k(1) + h_k(2)$ for all $k\geq m$. By Remark \ref{rmrk:ex_divisores}, we have that $\exp(2i\pi/p)$ is an eigenvalue of $(X,T)$ for its unique invariant measure. Then $\exp(2i\pi\alpha)$ is an eigenvalue of $(X,T)$ for every $\alpha\in\QQ$. 
In order to show that they are all non continuous it suffices to prove that $\operatorname{g.c.d.}(h_n(1),h_n(2))=1$ for all $n\geq 2$ (see Corollary \ref{coro:rootsofunity}).

From the proof of Corollary \ref{coro:ex_primeros_ev}, neither $2$ nor $3$ divide  $\operatorname{g.c.d.}(h_n(1),h_n(2))$ for all $n\geq 2$. Now take $m\geq 2$ and consider the $m$-th odd prime $p_m$. From \eqref{eq:ex_definiendo_sucesion_1ra} and \eqref{eq:ex_recursion_suma_alturas} we can see that $p_m$ divides $h_n(1) + h_n(2)$ for all $n\geq m+1$ and it does not do so for $n<m+1$.
Then, for $n<m+1$ it is not possible that $p_m$ divides $\operatorname{g.c.d.}(h_n(1),h_n(2))$.

Using \eqref{eq:ex_basic_relations_h} we have
\begin{align*}
&h_{m+1}(2) \\ 
& \quad = b_m\left(h_m(1) + h_m(2)\right) + h_m(2) \\
& \quad = 2b_m\left(b_{m-1} + 2\right)\left(h_{m-1}(1) + h_{m-1}(2)\right)  
 + b_{m-1}\left(h_{m-1}(1) + h_{m-1}(2)\right) + h_{m-1}(2) \\
\end{align*}
and by \eqref{eq:ex_definiendo_sucesion_1ra} and \eqref{eq:ex_definiendo_sucesion_2da} we get
\begin{align*}
h_{m+1}(2) 
& = \left(-3b_{m-1} - 8\right)\left(h_{m-1}(1) + h_{m-1}(2)\right) + h_{m-1}(2) \; \;  (\textrm{mod }p_m) \\
& \neq 0 \; \; (\textrm{mod }p_m).
\end{align*}

Then $p_m$ does not divide $h_{m+1}(2)$. So $p_m$ cannot divide $\operatorname{g.c.d.}(h_{m+1}(1),h_{m+1}(2))$.

A simple induction using \eqref{eq:ex_basic_relations_h} implies that $h_n=h_{m+1}\; (\textrm{mod } p_m)$ 
for $n > m+1$. So $p_m$ cannot divide $\operatorname{g.c.d.}(h_{n}(1),h_{n}(2))$ for $n>m+1$ either. 

We conclude that there is no prime integer dividing the $\operatorname{g.c.d.}(h_n(1),h_n(2))$ for all $n\geq 2$. This fact completes the proof.
\end{proof}

\begin{coro}\label{coro:ex_mezcladebiltopologica_y_muchos_vp}
There exists a topologically weakly mixing uniquely ergodic minimal Cantor system $(X,T)$ of finite topological rank whose group of eigenvalues for the unique invariant measure contains an isomorphic copy of $\QQ$.
\end{coro}
\begin{proof}
We consider the uniquely ergodic minimal Cantor system $(X,T)$ (of topological rank 2)  constructed in Proposition \ref{prop:ex_racionales_vpnocontinuos}. It has no roots of unity as continuous eigenvalues and 
its group of eigenvalues for the unique invariant measure contains an isomorphic copy of $\QQ$.

Using Corollary \ref{coro:erasingcontinuous} we have an strong orbit equivalent minimal Cantor system $(\tilde{X},\tilde{T})$ also of topological rank 2 and uniquely ergodic which is topologically weakly mixing and have the same measurable eigenvalues as $(X,T)$.
\end{proof}

\begin{rmrk}\label{rmrk:similar_example}
An example with similar characteristics can be obtained as follows. First we take any topological weakly mixing Cantor minimal  system of finite rank. Using \cite[Theorem~6.1]{ormes} we can prove that there exists a strong orbit equivalent system with all possible roots of unity as eigenvalues. Then this last system could have continuous eigenvalues but, by strong orbit equivalence, all of them need to be irrational. So, by applying Corollary \ref{coro:erasingcontinuous}, we can obtain a topological weakly mixing Cantor minimal system having all rational eigenvalues as non continuous ones. However, this process cannot guarantee that the resulting system is of topological finite rank or even expansive.
\end{rmrk}

\subsection{Minimal Cantor systems with the maximal continuous eigenvalue group property}

Let $(X,T)$ be a minimal Cantor system.
We set 
$$
E (X,T) = \{ \alpha  \in \RR | \exp (2i\pi \alpha ) \  \hbox{ is a continuous eigenvalue of }  (X,T) \} .
$$
We call it the {\em group of additive continuous eigenvalues}.
It is well-known that $E (X,T)$ is countable and contains $\ZZ$.
Let 
$$
I(X,T) = \bigcap_{\mu\in \mathcal{M} (X,T)} \left\{ \int_X f d\mu; f\in C(X,\mathbb{Z}) \right\},
$$ 
where $C(X,\ZZ)$ is the set of continuous functions from $X$ to $\ZZ$ and $\mathcal{M} (X,T)$ is the set of $T$-invariant probability measures. It is known that $I(X,T)$ is an invariant of strong orbit equivalence \cite{gps} and that $E(X,T)\subseteq I(X,T)$ \cite{Cortez&Durand&Petite:2016,GHH}. 

From Corollary \ref{coro:erasingcontinuous}, given a minimal Cantor system $(X,T)$ without rational continuous eigenvalues, there exists a strong orbit equivalent minimal Cantor system $(Y,S)$, so in particular $I(Y,S) = I(X,T)$, such that $E(Y,S)=\ZZ$ (this also can be deduced from \cite{ormes,GHH}).  
On the other extreme, for $(X,T)$ such that $E(X,T)\subsetneq I(X,T)$, it is not known whether one has a strong orbit equivalent system $(Y,S)$ such that $E(Y,S)=I(Y,S)=I(X,T)$ (for a deeper discussion and to motivate this question we refer the reader to \cite{Cortez&Durand&Petite:2016,GHH}). If the equality $E(Y,S)=I(Y,S)$ holds we say that $(Y,S)$ has the {\em maximal continuous eigenvalue group property}.

Below we provide a family of examples having this property using a result about the Brun algorithm for continued fractions and our criteria to be a continuous eigenvalue of a minimal Cantor system.

\subsubsection{Brun matrices and its properties}
In this section we present a version of a result due to A. Avila and V. Delecroix \cite{Avila&Delecroix:2017} that will help us to construct minimal Cantor systems having the maximal continuous eigenvalue group property. We will use and recall their notation.

We will make use of the following matrices coming from the so called Brun algorithm for multidimensional continued fractions (see \cite{Schweiger:2000}):
$$
B^{(1)} = \begin{pmatrix}1&1&0\\0&1&0\\0&0&1\end{pmatrix},
\qquad
B^{(2)} = \begin{pmatrix}1&1&0\\1&0&0\\0&0&1\end{pmatrix},
\qquad
B^{(3)} = \begin{pmatrix}1&0&1\\1&0&0\\0&1&0\end{pmatrix}.
$$
We call them Brun matrices and
for a word $w = w_1 \ldots w_n$ on the alphabet $\{ 1,2,3\}$
we set $B^{(w)} = B^{(w_1)} \cdots B^{(w_n)}$.

We say a non negative integer square matrix is \emph{Pisot} if its dominant eigenvalue is simple and all
the other eigenvalues have absolute values strictly less than one. 

\begin{prop} \cite{Brun:1957,Avila&Delecroix:2017}
\label{prop:brunprim}
Let $w$ be a word on the alphabet $\{ 1,2,3\}$.
Then, $B^{(w)}$ is primitive (some power of $B^{(w)}$ is strictly positive) 
if and only if $3$ appears as a symbol of $w$. 
Moreover, if $B^{(w)}$ is primitive then $B^{(w)}$ is Pisot.
\end{prop}

Let $C^{(c)}(0,1)$ ($C^{(r)}(0,1)$) denote the set of column (row) vectors $x \in \mathbb{R}^3$ with $\| x \| = 1$, where $\|\cdot \|$ is the supremum norm.  
For a real square matrix $M$ of dimension three and a subset of column vectors $R \subseteq \RR^3$ define  
$\|M\|_R  = \sup_{x \in R} \|M \|_x$, where 
$\|M\|_x = \sup_{\{z\in C^{(r)}(0,1); \ z\cdot x = 0\}} \|zM\|$. We are forced to distinguish between row and column vectors in order to be consistent with the notation used before. In particular, height and suffix vectors, $h_n$, $s_n$, $s_{m,n}$, etc., are row vectors and measure vectors, $\mu_n$, are column vectors.

Consider the matrices $A^{(1)}=B^{(1)}$, $A^{(2)}=B^{(2)}$ and $A^{(3)}=(B^{(3)})^6$. It is direct to verify that $A^{(3)}>0$. As before, for a word $w= w_1 \ldots w_n$ on the alphabet $\{ 1,2,3\}$ we define $A^{(w)}=A^{(w_1)} \cdots A^{(w_n)}$. 
By Proposition \ref{prop:brunprim}, if $w$ contains a $3$ then $A^{(w)}$ is Pisot (indeed, it is strictly positive). 

Let $D = \{x=(x_1,x_2,x_3)^T \in \RR^3 ; x_1> x_2 > x_3 >0 \}$ and for a word $w$ on the alphabet 
$\{1,2,3\}$ set $D^{(w)} = A^{(w)}D$. It is clear that $A^{(i)} D \subseteq D$ for all $i\in \{1,2,3\}$ and thus $D^{(w)} \subseteq D$. Moreover, it is a direct computation to verify that $A^{(w)}(\RR_{+}^3\setminus\{0\}) \subseteq D$ if $w$ contains at least one $3$, where $\RR_{+}$ is the set of nonnegative reals.

We will need the following three lemmas that give a finer structure of the products of Brun's matrices.
The first one is an adaptation from  \cite{Avila&Delecroix:2017} to matrices $A^{(w)}$.

\begin{lemma}
\label{lemme:cocycle}
Let $w$ and $w'$ be two words on the alphabet $\{ 1,2,3\}$. 
Then,
$$
\Vert A^{(w)}\Vert_{D^{(w)}} \leq 1 \hbox{ and } \Vert A^{(ww')}\Vert_{D^{(ww')}} \leq \Vert A^{(w)}\Vert_{D^{(w)}}  \Vert A^{(w')}\Vert_{D^{(w')}} .
$$
\end{lemma}
\smallskip

\begin{lemma}
\label{lem:brunpisot2}
Let $w$ be a word on the alphabet $\{ 1,2,3\}$ containing at least one $3$. 
Then, there exists an integer $n\geq 1$ such that  $||A^{(w^n)}||_{D^{(w^n)}} < 1$, where $w^n$ is the concatenation of the word $w$, $n$ times. 
\end{lemma}

\begin{proof}
From Proposition \ref{prop:brunprim} and the structure of the Brun matrices, $A^{(w)}$ is a Pisot matrix of determinant $1$ or $-1$. Since its characteristic polynomial is monic with integer coefficients, then it has three different roots $\alpha_1$, $\alpha_2$ and $\alpha_3$ that we can take
satisfying $\alpha_1 > 1 > |\alpha_2| \geq |\alpha_3|>0$. 
Let $y^{(1)}$, $y^{(2)}$ and $y^{(3)}$ be a base of $\RR^3$ formed by the corresponding right eigenvectors of $A^{(w)}$ and consider the matrix $P$ whose columns are $y^{(1)}$, $y^{(2)}$ and $y^{(3)}$. 
Similarly, let $z^{(1)}$, $z^{(2)}$ and $z^{(3)}$ be left eigenvectors of $A^{(w)}$ associated to $\alpha_1$, $\alpha_2$ and $\alpha_3$ and consider the matrix $Q$ whose rows are $z^{(1)}$, $z^{(2)}$ and $z^{(3)}$. Clearly, $P$ and $Q$ are invertible and $ z^{(i)}\cdot y^{(j)}=0$ if $i\not = j$. 
We can also take previous eigenvectors satisfying $z^{(i)}\cdot y^{(i)}=1$ and $\| z^{(i)}\| =1$ for 
all $i\in \{1,2,3\}$. By continuity, there exist $\delta_1>0$ and $\delta_2 >0$ such that 
\begin{equation}
\label{eq:deltas}
\sup_{x\in P^{-1} C^{(c)}(0,1)} \Vert x\Vert\leq \delta_2, \ 
\sup_{z\in Q^{-1} C^{(r)}(0,1)} \Vert z \Vert\leq \delta_2 \ \text{ and } \
\inf_{x\in P^{-1} (D \cap C^{(c)}(0,1))} |x_1| \geq \delta_1.
\end{equation}
\smallskip

Let $n\geq 1$ and $y=A^{(w^n)} y' \in D^{(w^n)}$ with $y' \in D \cap C^{(c)}(0,1)$. 
By definition of $\Vert A^{(w^n)}\Vert_{D^{(w^n)}}$ it is enough to consider normalized vectors like $y'$.
Since $y' = a_1 y^{(1)} + a_2 y^{(2)} + a_3 y^{(3)}=P (a_1,a_2,a_3)^T$ we have that 
$(a_1,a_2, a_3)^T \in P^{-1} (D\cap C^{(c)}(0,1))$. Now, take $z\in C^{(r)}(0,1)$ such that $z \cdot y = 0$. 
As before, we have that 
$z = b_1 z^{(1)} + b_2 z^{(2)} + b_3 z^{(3)}$ with $(b_1,b_2, b_3)$ in $Q^{-1} C^{(r)}(0,1)$.
From \eqref{eq:deltas} we obtain $|a_1|\geq \delta_1$, $|a_i|\leq \delta_2$ and 
$|b_i|\leq \delta_2$ for $i\in \{1,2,3\}$. Also, 
$$
0 = z \cdot y  = z A^{(w^n)} y'   
=
a_1b_1 \alpha_1^n  + a_2b_2 \alpha_2^n  + a_3b_3 \alpha_3^n .  
$$
Hence,
\begin{align*}
\Vert z A^{(w^n)} \Vert & = \Vert b_1  \alpha_1^n z^{(1)}+ b_2 \alpha_2^n z^{(2)} + b_3 \alpha_3^n z^{(3)}\Vert \\
& =
\left \Vert \left(
- \frac{a_2}{a_1}b_2\alpha_2^n  -  \frac{a_3}{a_1}b_3\alpha_3^n
 \right)
z^{(1)}
+ b_2 \alpha_2^n z^{(2)} + b_3 \alpha_3^n z^{(3)} \right \Vert\\
& \leq 
\left (\left |\frac{a_2}{a_1}b_2 \right| +  \left |\frac{a_3}{a_1}b_3 \right |
+| b_2 |+ |b_3  | \right)|\alpha_2|^n
\leq 
2 \left (\frac{\delta_2}{\delta_1}  
+ 1 \right )\delta_2 |\alpha_2|^n.
\\
\end{align*}
Taking $n$ large enough we conclude. 
\end{proof}

\begin{lemma}
\label{lemma:ratind}
Let ${\bf w}=({\bf w}_n)_{n\geq 0}$ be a sequence in $\{ 1,2,3 \}^\mathbb{N}$ where $3$ appears infinitely many times. Assume there is an increasing sequence of positive integers $(n_j)_{j\geq 0}$ such that $\displaystyle\lim_{j\to +\infty} \| A^{({\bf w}_{[0,n_j)})} \|_{ D^{({\bf w}_{[0,n_j)})} } = 0$, where ${\bf w}_{[0,n_j)}={\bf w}_0\ldots {\bf w}_{n_j-1}$. 
Then, any $\nu \in \bigcap_{n\geq 1} A^{({\bf w}_{[0,n)})} D$ has rationally independent entries. 
\end{lemma}

\begin{proof}
On the contrary, take a nonzero integer row vector $z$ such that $z\cdot\nu =0$.
Since $A^{({\bf w}_{[0,n)})}$ is invertible, then $z A^{({\bf w}_{[0,n)})}$ is a nonzero integer vector for all $n\geq 1$ and 
\begin{align*}
0 < \frac{1}{\Vert z\Vert } \leq \| \frac{z}{\Vert z \Vert} A^{({\bf w}_{[0,n)})}   \| \leq \| A^{({\bf w}_{[0,n)})} \|_\nu \leq  \| A^{({\bf w}_{[0,n)})} \|_{ D^{({\bf w}_{[0,n)})} }, 
\end{align*}
where in the last inequality we have used that $\nu \in \bigcap_{n\geq 1} A^{({\bf w}_{[0,n)})} D=
\bigcap_{n\geq 1} D^{({\bf w}_{[0,n)})} $.
Taking liminf in last expression leads to a contradiction. 
\end{proof}

\subsubsection{Constructing a minimal Cantor system having the maximal continuous eigenvalue group property}
In this section we apply previous results to define a family of minimal Cantor systems 
$(X,T)$ such that $I(X,T)=E(X,T)$, {\it i.e.}, satisfying the maximal continuous eigenvalue group property. 

\begin{prop}
\label{prop:example}
Let ${\bf w}=({\bf w}_n)_{n\geq 0}$  be a sequence in $\{ 1,2,3\}^{\mathbb{N}}$ where $3$ appears infinitely many times and
let $(n_j)_{j\geq 0}$ be an increasing sequence of positive integers  such that
\begin{align}
\label{hyp:CV}
\sum_{j\geq 0}  (n_{j+1}-n_j)\| A^{({\bf w}_{[0,n_{j})})} \|_{D^{({\bf w}_{[0,n_{j})})}} < +\infty.
\end{align} 
Let $(X,T)$ be a finite rank minimal Cantor system given by a Bratteli-Vershik representation  
whose incidence matrices are $M_1=h_1=(1,1,1)$ and $M_n=A^{({\bf w}_{n-2})}$ for $n\geq 2$.  
Then, $(X,T)$ is uniquely ergodic and 
there exists a real vector $\nu = (\nu (1) , \nu (2) , \nu (3))^T \in D$ with rationally independent entries such that 
$$E(X,T) = I (X,T) =  \nu (1) \mathbb{Z} + \nu (2) \mathbb{Z} + \mathbb{Z}.$$
\end{prop}

We notice that, by the choice of the incidence matrices, a Bratteli diagram as the one described in this proposition is always simple (this is just the fact that matrix $A^{(3)}$ is strictly positive) and it always admits a local order which makes it properly order. It is enough to consider the so called left-right order infinitely many times. 

\begin{proof}
Under our assumptions the incidence matrices of the Bratteli diagram are given by 
$M_n = A^{({\bf w}_{n-2})}$, $P_{n}= A^{({\bf w}_{[0,n-1)})}$ and 
$P_{m,n}=A^{({\bf w}_{[m-1,n-1)})}$ for all $1 \leq m < n$.

Let $\mu$ be an ergodic measure of $(X,T)$. For all $1\leq m < n$ we have that $\mu_1=P_m\mu_m$ and 
$\mu_m=P_{m,n}\mu_n$. Since  $A^{(w)}(\RR_+^3\setminus\{0\}) \subseteq D$ 
for any word on the alphabet $\{1,2,3\}$ having a $3$, we deduce that 
$\mu_n \in D$ for all $n\geq 2$ and thus $\mu_1 \in \bigcap_{n\geq 2} A^{({\bf w}_{[0,n-1)})} D$.

For $i\in \{ 1,2,3\}$ define the row vector 
$\eta_i = \mu_1 (i) h_1 -e_i$, where $e_i$ is the $i$-th canonical row vector of $\RR^3$. 
From $\mu_1 (1 ) + \mu_1 (2) + \mu_1 (3) = 1$ we get 
$\eta_i \cdot \mu_1 = 0$. Thus, by definition, for all $i \in \{1,2,3\}$ and $n\geq 2$  we have 
$$\Vert \frac{\eta_i}{\Vert \eta_i \Vert} A^{({\bf w}_{[0,n-1)})}  \Vert \leq \Vert A^{({\bf w}_{[0,n-1)})} \Vert_{\mu_1}  \leq \Vert A^{({\bf w}_{[0,n-1)})} \Vert_{D^{({\bf w}_{[0,n-1)})}}, $$
where in the last inequality we have used that $\mu_1 \in A^{({\bf w}_{[0,n-1)})} D$. 
Hence, for any $n\in [n_j+1, n_{j+1} + 1)$ and $s \in S_n(u,v)$ with $u\in V_n$ and $v\in V_{n+1}$ 
we have
\begin{align*}
\vvert \mu_1 (i)  \langle  s ,h_n  \rangle\vvert  
= &
\vvert \mu_1 (i)  \langle  s , h_1 P_n \rangle\vvert
=
\vvert \mu_1 (i)  \langle  s , h_1 A^{({\bf w}_{[0,n-1)})} \rangle\vvert
=
\vvert  \langle  s , \eta_i A^{({\bf w}_{[0,n-1)})} \rangle\vvert\\
\leq 
& \Vert \eta _i A^{({\bf w}_{[0,n-1)})}  \Vert \ \Vert s \Vert \leq  
\Vert A^{({\bf w}_{[0,n-1)})} \Vert_{D^{({\bf w}_{[0,n-1)})}} \ \Vert \eta_i \Vert \ \Vert s \Vert \\
\leq & \Vert A^{({\bf w}_{[0,n_j)})} \Vert_{D^{({\bf w}_{[0,n_j)})}} \  \Vert s \Vert,
\end{align*}
where in the last inequality we have used Lemma \ref{lemme:cocycle} and the fact that 
$\Vert \eta_i \Vert \leq 1$.
But the set of incidence matrices we are using is bounded, so 
$\max \{ \Vert s \Vert ;  s \in S_n(u,v) , u\in V_n, v\in V_{n+1} \} \leq L$, where $L$ is a universal constant. This inequality implies that the series 
$$
\displaystyle \sum_{n\geq 1} \max_{\small \begin{array}{l}s \in S_n(u_n,u_{n+1})\\u_n\in V_n, u_{n+1}\in V_{n+1}\end{array}} \vvert \mu_1(i)  \langle  s ,h_n  \rangle\vvert  
$$
is bounded by $L \cdot\sum_{j\geq 0}  (n_{j+1} - n_j )\| A^{({\bf w}[0,n_{j})])} \|_{D^{({\bf w}[0,n_{j})])}}$ and thus, by hypothesis, it converges. Then, by Corollary \ref{rmrk:cond_suf_sin_simple}, 
$\mu_1 (1)$, $\mu_1 (2)$ and $\mu_1 (3)$  are continuous eigenvalues of $(X,T)$. 

By hypothesis and since $\mu_1 \in \bigcap_{n\geq 2} A^{({\bf w}_{[0,n-1)})} D$ for all $n\geq 2$, from Lemma \ref{lemma:ratind} we conclude that $\mu_1 (1)$, $\mu_1 (2)$ and $\mu_1 (3)$
are rationally independent continuous eigenvalues of $(X,T)$. 
Consequently, $1$, $\mu_1 (1)$ and $\mu_1 (2)$ are rationally independent too. 
This shows that $\mu_1 (1) \mathbb{Z} + \mu_1 (2) \mathbb{Z} + \mathbb{Z} \subseteq E(X,T)$.

By Theorem 9 in \cite{rangofinito}, the number of ergodic measures of $(X,T)$ is bounded by above by 
$\displaystyle \min_{n\geq 1}|V_n|-\eta(X,T)+1=4-\eta(X,T)$, where  $\eta(X,T)$ is the maximal number of rationally independent additive continuous eigenvalues. But, since $1$, $\mu_1 (1)$ and $\mu_1 (2)$ are rationally independent, this bound is lower than 1. We conclude that $(X,T)$ is uniquely ergodic. 

Finally, by unique ergodicity and the fact that $\mu_n=P^{-1}_n\mu_1$ for all $n\geq 2$, we have that 
$I(X,T) \subseteq \mu_1 (1) \mathbb{Z} + \mu_1 (2) \mathbb{Z} + \mathbb{Z} \subseteq E(X,T)$. As $E(X,T)\subseteq I(X,T)$ we get that $E(X,T)=I(X,T)=\mu_1(1) \mathbb{Z} + \mu_1(2) \mathbb{Z} + \mathbb{Z}$ as desired. 
\end{proof}

Let us explain how to construct sequences satisfying previous proposition. 
Let $w$ be a word in $\{ 1,2,3\}$ having at least one occurrence of $3$.
From Proposition \ref{prop:brunprim}, $A^{(w)}$ is Pisot and Lemma \ref{lem:brunpisot2} implies that  
$\delta = ||A^{(w^n)}||_{D^{(w^n)}} < 1$ for some $n\geq 1$.
Let ${\bf w}\in \{ 1,2,3\}^\NN$ and $(n_j)_{j\geq 0}$ be an increasing sequence of positive integers with $n_0=1$ and  
\begin{enumerate}
\item
${\bf w}_{[n_j , n_j +n|w|)} = w^n$ and $n_{j+1}-n_j > n|w|$ for all $j\geq 0$;
\item
$\sum_{j\geq 0} \delta^j (n_{j+1} - n_j ) < +\infty$.
\end{enumerate} 
Then $\bf w$ satisfies hypothesis of Proposition \ref{prop:example} and any minimal Cantor system $(X,T)$ 
satisfying the conditions of this proposition has the maximal continuous eigenvalue group property. 
In addition, making some modifications in previous construction we can get that the set of sequences like 
$\bf w$ can be taken to have full measure for many shift invariant measures of $\{1,2,3\}^\NN$.

{\small
\bibliography{biblio}{}
\bibliographystyle{amsalpha}
}

\end{document}